\documentclass[12pt,reqno]{article}
\usepackage{amsmath,amsthm,amsfonts,amssymb,bbm}
\usepackage{graphicx}
\usepackage{pstricks,pst-plot}
\usepackage[english]{babel}
\usepackage[cp1251]{inputenc}

\theoremstyle{plain}

\numberwithin{equation}{section}

\newtheorem {lemma} {Lemma}[section]
\newtheorem {theorem} {Theorem}[section]
\newtheorem{proposition}{Proposition}[section]
\theoremstyle{remark}

\DeclareMathOperator{\Dim}{Dim}
\DeclareMathOperator{\tr}{tr}

\begin{document}

\title{Plancherel representations of $U(\infty)$ and correlated Gaussian Free Fields}

\author{Alexei Borodin\thanks{Massachusetts Institute of Technology, Cambridge, USA, and Institute for Information Transmission Problems, Moscow, Russia},
Alexey Bufetov\thanks{Higher School of Economics, Moscow, Russia, and Institute for Information Transmission Problems, Moscow, Russia}
}
\date{}

\maketitle

\begin{abstract}
We study asymptotics of traces of (noncommutative) monomials
formed by images of certain elements of the universal enveloping algebra of the infinite-dimensional
unitary group in its Plancherel representations. We prove that they converge to
(commutative) moments of a Gaussian process that can be viewed as a collection of
simply yet nontrivially correlated two-dimensional Gaussian Free Fields. The limiting
process has previously arisen via the global scaling limit of spectra for
submatrices of Wigner Hermitian random matrices.

The results of the present work were announced in \cite{BB-ann}.
\end{abstract}

\section{Introduction}

Asymptotic studies of measures on partitions of representation theoretic origin
is a well-known and popular subject. In addition to its intrinsic importance in
representation theory, see e.g. \cite{BO-European} and references therein,
it enjoys close connections to the theory of random
matrices, interacting particle systems, enumerative combinatorics, and other domains,
for which it often provides crucial technical tools, cf. e.g.
\cite{Ok}, \cite{Olsh-Oxford}, \cite{BG}.

A typical scenario of how such measures arise is as follows: One starts
with a group that has a well known list of irreducible representations, often parametrized
by partitions or related objects. Then a decomposition of a natural reducible representation
of this group on irreducibles provides a split of the total dimension of the
representation space into dimensions of the corresponding isotypical components;
their relative sizes are the weights of the measure. This procedure is well-defined
for finite-dimensional representations, but also for infinite-dimensional representations
with finite trace; the weight of (the label of) an isotypical component is then defined
as the trace of the projection operator onto it, provided that the trace is normalized
to be equal to 1 on the identity operator.

An alternative approach to measures of this sort consists in defining
averages with respect to such a measure for a suitable set of functions on labels
of the irreducible representations. These averages are obtained as traces of the
operators in the ambient representation space that are scalar in each of the
isotypical components. In their turn, the operators are images of central elements
in the group algebra of the group if the group is finite, or in the universal
enveloping algebra of the Lie algebra if one deals with a Lie group. The central
elements form a commutative algebra that is being mapped to the algebra of functions
on the labels, i.e., on partitions or their relatives. The value of the function
corresponding to a central element at a representation label is the (scalar) value
of this element in that representation.

While one may be perfectly satisfied with such an approach from the probabilistic point of
view, from the representation theoretic point of view it is somewhat unsettling that we
are able to only deal with commutative subalgebras this way, while the main interest
of representation theory is in noncommutative effects.

The goal of this work is go beyond this commutativity constraint.

More exactly, in a specific setting of the finite trace representations of the
infinite-dimensional unitary group $U(\infty)$ described below, we consider a family of commutative
subalgebras of the universal enveloping algebra such that elements from different
subalgebras generally speaking do not commute. We further consider the limit
regime in which the measures for each of the commutative subalgebras are known
to approximate the two-dimensional Gaussian Free Field (GFF), see \cite{BorFer}.
We want to study the ``joint distribution'' of these GFFs for different subalgebras,
whatever this might mean.

For any element of the universal enveloping algebra, one can define its ``average''
as the trace of its image in the representation. Thus, having a representation,
we can define ``averages'' for arbitrary products of elements from our subalgebras,
despite the fact that the elements do not commute.

Our main result (fully stated in Section \ref{sc:mainresult} below) is that for certain {\it Plancherel\/} representations, these ``averages''
converge to actual averages of suitable
observables on a Gaussian process that consists of a family of explicitly correlated
GFFs. Thus, the original absence of commutativity in this limit disappears, and yet
the limiting GFFs that arise from different commutative subalgebras do not become
independent.

The same limiting object (the collection of correlated GFFs) has been previously
shown to be the universal global scaling limit for eigenvalues of
various submatrices of Wigner Hermitian random matrices, cf. \cite{Bor, Bor2}.
We also expect it to arise from other, non-Plancherel factor representations of
the infinite-dimensional unitary group under appropriate limit transitions.

The paper is organized as follows.

In Section \ref{sc:preliminaries} we give necessary definitions and collect a few known facts on asymptotics of the Plancherel measures for $U(\infty)$.
Section \ref{sc:statement} contains the statement of our main result.
In Section \ref{sc:covariance} we show how one obtains the variance of the limiting Gaussian process.
The final Section \ref{sc:proof} contains a proof of the asymptotic normality of the chosen
``observables''.

\subsubsection*{Acknowledgements}
The authors are very grateful to Grigori Olshanski for numerous discussions that were
extremely helpful. A.~Borodin was partially supported by NSF grant DMS-1056390.
A.~Bufetov was partially supported by Simons Foundation-IUM scholarship, by Moebius
Foundation for Young Scientists, and by RFBR--CNRS grant 10-01-93114.

\section{Preliminaries}\label{sc:preliminaries}

\subsection{The infinite-dimensional unitary group and its characters}
\label{infinite group}

Let $U(N)=\left\{ [u_{ij}]_{i,j=1}^N \right\}$ be the group of $N \times N$ unitary matrices.
Consider the tower of embedded unitary groups
\begin{equation*}
U(1) \subset U(2) \subset \dots U(N) \subset U(N+1) \subset \dots,
\end{equation*}
where the embedding $U(k) \subset U(k+1)$ is defined by $u_{i,k+1}=u_{k+1,i}=0$, $1 \le i \le k$, $u_{k+1,k+1}=1$. \emph{The infinite--dimensional unitary group } is the union of these groups:

\begin{equation*}
U(\infty) = \bigcup_{N=1}^{\infty} U(N).
\end{equation*}

Define a \emph{character} of the group $U(\infty)$ as a function
$\chi : U(\infty) \to \mathbb C$ that satisfies

1) $\chi(e)=1$, where $e$ is the identity element of $U(\infty)$ (normalization);

2) $\chi(g h g^{-1}) = \chi (h)$, where $g,h$ are any elements of $U(\infty)$ (centrality);

3) $[\chi( g_i g_j^{-1})]_{i,j=1}^n$ is an Hermitian and positive-definite matrix for
any $n\ge 1$ and $g_1, \dots, g_n \in U(\infty)$ (positive-definiteness);

4) the restriction of $\chi$ to $U(N)$ is a continuous function for any $N\ge 1$
(continuity).

The space of characters of $U(\infty)$ is obviously convex. The extreme points of this space are called \emph{extreme} characters; they replace irreducible characters in this setting. The classification of the extreme characters is known as the Edrei--Voiculescu theorem (see \cite{Vo}, \cite{Edr}, \cite{VK}, \cite{OkoOls98}, \cite{BorOls}).
It turns out that the extreme characters can be parameterized by the set $\Omega=(\alpha^+, \alpha^-, \beta^+,\beta^-, \delta^+, \delta^-)$, where

\begin{gather*}
\alpha^{\pm} = \alpha_1^{\pm} \ge \alpha_2^{\pm} \ge \dots \ge 0, \\
\beta^{\pm} = \beta_1^{\pm} \ge \beta_2^{\pm} \ge \dots \ge 0, \\
\delta^{\pm} \ge 0, \ \ \ \sum_{i=1}^{\infty} (\alpha_i^{\pm}+\beta_i^{\pm}) \le \delta^{\pm}, \ \ \ \beta_1^+ + \beta_1^- \le 1.
\end{gather*}

Instead of $\delta^{\pm}$ we can use parameters $\gamma^{\pm}\ge 0$ defined by
\begin{equation*}
\gamma^{\pm} := \delta^{\pm} - \sum_{i=1}^{\infty} (\alpha_i^{\pm} + \beta_i^{\pm}).
\end{equation*}

Each $\omega \in \Omega$ defines a function $f_0^{\omega} : \{ u \in \mathbb C : |u|=1 \} \to \mathbb C$ by
\begin{equation*}
f_0^{\omega} (u) = \exp( \gamma^+ (u-1) + \gamma^{-} (u^{-1}-1) ) \prod_{i=1}^{\infty}
\frac{(1+\beta_i^+ (u-1))}{(1-\alpha_i^+ (u-1))} \frac{(1+\beta_i^- (u^{-1}-1))}{(1- \alpha_i^- (u^{-1}-1))}.
\end{equation*}
Let
\begin{equation*}
\chi^{\omega} (U) := \prod_{u \in Spectrum(U)} f_0^{\omega} (u), \qquad U \in U(\infty).
\end{equation*}
Then $\chi^{\omega}$ is the extreme character of $U(\infty)$ corresponding to $\omega \in \Omega$.

A \emph{signature} (also called \emph{highest weight}) of length $N$ is a sequence of $N$ weakly decreasing integers
\begin{equation*}
\lambda = \left( \lambda_1 \ge \lambda_2 \ge \dots \ge \lambda_N \right), \qquad \lambda_i \in \mathbb Z, \quad 1\le i\le N.
\end{equation*}

It is well known that the irreducible (complex) representations of $U(N)$ can be parametrized by signatures of length $N$ (see e.g. \cite{Wey39}, \cite{Zhe}).
Let $\Dim_N (\lambda)$ be the dimension of the representation corresponding to $\lambda$. By $\chi^{\lambda}$ we denote the conventional character
of this representation (i.e., the function on the group obtained by evaluating the trace of the
representation operators) divided by $\Dim_N (\lambda)$.

Represent a signature $\lambda$ as a pair of Young diagrams $(\lambda^+, \lambda^-)$, where $\lambda^+$ consists of non-negative $\lambda_i$'s and $\lambda^-$ consists of negative $\lambda_i$'s:
\begin{equation*}
\lambda = (\lambda_1^+, \lambda_2^+, \dots, -\lambda_2^-, -\lambda_1^-).
\end{equation*}

For Young diagram $\mu$ let $|\mu|$ be the number of boxes of $\mu$, and let $d(\mu)$ be the number of diagonal boxes of $\mu$. Let $d(\lambda^+)=d^+$ and $d(\lambda^-)=d^-$. Define the \emph{modified Frobenius coordinates} of $\mu = (\mu_1, \dots , \mu_k)$ by setting
\begin{equation}
\label{frob}
a_i = \mu_i - i +\frac12, \ \ \  b_i = \mu'_i -i + \frac12, \qquad 1 \le i \le d(\mu),
\end{equation}
where $\mu'$ is the transposed diagram.

Given a sequence $\{ f_N \}$ of functions $f_N: U(N) \to \mathbb C$ and a function $f: U(\infty) \to \mathbb C$, we say that $f_N$'s \emph{approximate} $f$ if for any fixed $N_0 \in \mathbb N$ the restrictions of the functions $f_N$ to $U(N_0)$ uniformly converge to the restriction of $f$ to $U(N_0)$ as $N\to\infty$.

It turns out that the extreme characters of $U(\infty)$ can be approximated by (normalized) irreducible characters of $U(N)$.

\begin{theorem}
Denote by $\chi^{\omega}$ the extreme character of $U(\infty)$ corresponding to $\omega = (\alpha^{\pm}, \beta^{\pm}, \delta^{\pm}) \in \Omega$. Let $\{ \lambda(N) \}$ be a sequence of signatures of length $N$ such that the (modified) Frobenius coordinates of $\lambda^{\pm}$ are equal to $a^{\pm}_i (N)$, $b^{\pm}_i (N)$. Then the characters $\chi^{\lambda(N)}$ approximate $\chi^{\omega}$ iff

\begin{equation*}
\lim_{N \to \infty} \frac{a_i^{\pm} (N)}{N} = \alpha_i^{\pm}, \ \ \
\lim_{N \to \infty} \frac{b_i^{\pm} (N)}{N} = \beta_i^{\pm}, \ \ \
\lim_{N \to \infty} \frac{|\lambda^{\pm} (N)|}{N} = \delta^{\pm}.
\end{equation*}

\end{theorem}

\begin{proof}
This theorem is due to Vershik and Kerov, see \cite{VK}. See \cite{OkoOls98} and \cite{BorOls} for detailed proofs.
\end{proof}

\subsection{The Gelfand-Tsetlin graph and coherent systems of measures}
\label{gelfand-tsetlin}

Let $\mathbb {GT}_N$ denote the set of all signatures of length $N$. (Here the letters $\mathbb{GT}$ stand for
`Gelfand-Tsetlin'.)
We say that $\lambda \in \mathbb {GT}_N$ and $\mu \in \mathbb{GT}_{N-1}$ \emph{interlace},
notation $\mu \prec \lambda$, iff $\lambda_i \ge \mu_i \ge \lambda_{i+1}$ for any $1 \le i \le N-1$. We also define $\mathbb{GT}_0$ as a singleton consisting of an element that we denote as
$\varnothing$. We assume that $\varnothing \prec \lambda$ for any $\lambda\in \mathbb{GT}_1$.

The \emph{Gelfand-Tsetlin graph} $\mathbb {GT}$ is defined by specifying its set of vertices as
$\bigcup_{N=0}^{\infty} \mathbb{GT}_N $ and putting an edge between any two signatures $\lambda$ and $\mu$ such that
either $\lambda \prec \mu$ or $\mu \prec \lambda$.
A \emph{path} between signatures $\kappa \in \mathbb {GT}_K$ and $\nu \in \mathbb {GT}_N$, $K<N$, is a sequence
\begin{equation*}
\kappa = \lambda^{(K)} \prec \lambda^{(K+1)} \prec \dots \prec \lambda^{(N)} = \nu, \qquad  \lambda^{(i)} \in \mathbb{GT}_i,\quad K\le i\le N.
\end{equation*}

It is well known that $\Dim_N (\nu)$ is equal to the number of paths between $\varnothing$ and $\nu \in \mathbb {GT}_N$. An \emph{infinite path} is a sequence
\begin{equation*}
\varnothing \prec \lambda^{(1)} \prec \lambda^{(2)} \prec \dots \prec \lambda^{(k)} \prec \lambda^{(k+1)} \prec \dots.
\end{equation*}

We denote by $\mathcal P$ the set of all infinite paths. It is a topological space with
the topology induced from the product topology on the ambient product of discrete sets
$\prod_{N\ge 0}\mathbb{GT}_N$. Let us equip $\mathcal P$ with the Borel $\sigma$-algebra.

For $N=0,1,2,\dots$, let $ M_N$ be a probability measure on $\mathbb {GT}_N$. We say that
$\{ M_N \}_{N=0}^{\infty}$ \emph{is a coherent system of measures} if for any $N\ge 0$
and $\lambda\in\mathbb{GT}_N$,
\begin{equation*}
M_{N} (\lambda) = \sum_{\nu : \lambda \prec \nu} M_{N+1} (\nu) \frac{\Dim_{N} (\lambda)}{\Dim_{N+1} (\nu)}.
\end{equation*}
Given a coherent system of measures $\{ M_N \}_{N=1}^{\infty}$, define the weight of a
cylindric set of $\mathcal P$ consisting of all paths with prescribed members up to $\mathbb{GT}_N$ by
\begin{equation}
\label{mera-puti}
P ( \lambda^{(1)}, \lambda^{(2)}, \dots, \lambda^{(N)} ) = \frac{ M_N (\lambda^{(N)})}{\Dim_N (\lambda^{(N)} )}.
\end{equation}
Note that this weight depends on $\lambda^{(N)}$ only (and does not depend on $\lambda^{(1)}$, $\lambda^{(2)}$, $\dots$, $\lambda^{(N-1)}$). The coherency property implies that
these weights are consistent, and they correctly define a Borel probability measure on $\mathcal P$.

Let $\chi$ be a character of $U(\infty)$. It turns out that for any $N\ge 1$, its restriction to $U(N)$ can be decomposed into a series in $\chi^{\lambda}$,
\begin{equation}
\label{char}
{\chi |}_{U(N)} = \sum_{\lambda \in \mathbb{GT}_N} M_N (\lambda) \chi^{\lambda}.
\end{equation}
It is readily seen that the coefficients $M_N (\lambda)$ form a coherent system of measures on $\mathbb {GT}$.
Conversely, for any coherent system of measures on $\mathbb{GT}$ one can construct a character
of $U(\infty)$ using the above formula.

\subsection{The probability measure corresponding to the one-sided Plancherel character}
\label{probMeasure}

Let $\chi^{\gamma^+}$ be an extreme character of $U(\infty)$ corresponding to parameters $\alpha^{+,-}=0$, $\beta^{+,-}=0$, $\gamma^- =0 $, and nonzero $\gamma^+$. By analogy with the classification of extreme characters of the infinite symmetric group, this character is called the \emph{one-sided Plancherel} character. Denote by $\tilde P^{\gamma^+}_N $ the coherent system of measures on ${\mathbb {GT} }$ corresponding to $\chi^{\gamma^+}$.

Let
\begin{equation*}
P_L^{\gamma} (\lambda) := \tilde P^{\gamma L}_L (\lambda), \qquad L \in \mathbb N,
\end{equation*}
where $\gamma>0$ is a fixed constant.

Suppose $S(n)$ is the symmetric group of degree $n$, $\mathbb Y_n$ is the set of Young diagrams with $n$ boxes, and $\dim \mu$ is the dimension of the irreducible representation of $S(n)$ corresponding to $\mu \in \mathbb Y_n$.

Let $u_1, \dots, u_L$ be the eigenvalues of the matrix $U \in U(L)$.

In order to obtain an explicit formula for $P^{\gamma}_L (\lambda)$ we need (see \eqref{char}) to decompose the function
\begin{equation*}
f_0^{\gamma L} (u_1, \dots, u_L) = \exp \left( \gamma L \sum_{i=1}^L \left(u_i-1 \right) \right)
\end{equation*}
on normalized irreducible characters of $U(L)$. It is well known (see e.g. \cite{Wey39}, \cite{Zhe}, \cite{GooWal}) that they are defined by
\begin{equation*}
\chi^{\lambda} = \dfrac{s_{\lambda} (u_1, \dots, u_L)}{\Dim_L \lambda},
\end{equation*}
where $s_{\lambda}$ is the Schur function (see e.g. \cite{Mac} for a definition).

Let us write the function $f_0^{\gamma L} (u_1, \dots, u_L)$ in the form
\begin{equation*}
\exp \left(\gamma L \sum_{i=1}^L (u_i-1) \right) = \exp(-\gamma L^2) \sum_{n=1}^{\infty}
\frac{ (\gamma L)^n p_1^n (u_1, \dots, u_L)}{n!},
\end{equation*}
where $p_1 (u_1, \dots, u_L) := \sum_{i=1}^L u_i$.

Using the well-known formula (see \cite{Mac})
\begin{equation*}
p_1^n (u_1, \dots, u_L) = \sum_{ \lambda \in \mathbb Y_n(L) } \dim \lambda \cdot s_{\lambda},
\end{equation*}
where $\mathbb Y_n (L)$ is the set of Young diagrams with $n$ boxes and no more than $L$ rows, we obtain
\begin{equation}
\label{prec1}
P^{\gamma}_L (\lambda) = \begin{cases}
e^{- \gamma L^2} \dfrac{ (\gamma L)^{\lambda_1 + \dots + \lambda_L}}{(\lambda_1 + \dots + \lambda_L)!} \dim \lambda
\Dim_L \lambda, &\text{ if } \lambda_1 \ge \dots \ge \lambda_L \ge 0; \\
0, &\text{ otherwise.}
\end{cases}
\end{equation}
Therefore, $P_L^{\gamma}$ is supported by signatures with non-negative coordinates or, equivalently, by Young diagrams with no more than $L$ rows.

Let us introduce another family of measures on Young diagrams previously considered by Biane \cite{Biane}, which is closely related to $\{ P^{\gamma}_L \}$. Let $N$ and $n$ be two positive integers. Consider the tensor space $V=(\mathbb C^N)^{\otimes n}$ as a bimodule with respect to the natural commuting actions of the groups $S(n)$ and $U(N)$. By the Schur-Weyl duality, the representation of the group $S(n) \times U(N)$ in $V=(\mathbb C^N)^{\otimes n}$ has simple spectrum which is indexed by Young diagrams $\lambda \in \mathbb Y_n (N)$. The dimension of the irreducible representation corresponding to $\lambda$ equals $\dim \lambda \cdot \Dim_N \lambda$. This serves as a prompt for introducing a probability measure $M_{n,N}^{SW} (\lambda)$ on $\mathbb Y_n (N)$:
\begin{equation*}
M_{n,N}^{SW} (\lambda) = \frac{\dim \lambda \Dim_N \lambda}{N^n}, \qquad \lambda \in \mathbb Y_n (N).
\end{equation*}

Let us substitute for $n$ a Poisson random variable with parameter $\nu$; then we obtain the
\emph{Poissonization} of the measures $M_{n,N}^{SW}$:
\begin{equation}
\label{SWP}
M_{\nu, N}^{SWP} (\lambda) = e^{-\nu} \frac{\nu^{|\lambda|}}{|\lambda|!} M_{|\lambda|,N}^{SW} (\lambda),
\qquad \lambda \in \mathbb Y.
\end{equation}
From \eqref{prec1} and \eqref{SWP} we have
\begin{equation*}
M_{\gamma L^2, L}^{SWP} (\lambda) = P_L^{\gamma} (\lambda).
\end{equation*}

The Poisson random variable with large parameter $\nu$ is concentrated around $\nu$. Therefore, asymptotic properties of $M^{SW}_{n,N}$ are close to asymptotic properties of $P_L^{\gamma}$ for $L=N$ and $n=[\gamma L^2]$.

As was shown in \cite{BorOls99}, \cite{Joh}, \cite{BorOlsh}, the random Young diagram distributed according to $P_L^{\gamma}$ gives rise to a determinantal random point process; this process is the Charlier orthogonal polynomial ensemble.

\subsection{Known results about Plancherel measures}
In this section we review some known results about $P^{\gamma}_L (\lambda)$ and $M_{[\gamma L^2], L}^{SW} (\lambda)$.

Take a Young diagram $\lambda$, flip and rotate it 135 degrees and denote by $\lambda(x) : \mathbb R \to \mathbb R$ the continuous piecewise linear function corresponding to the upper boundary of $\lambda$ with the condition $\lambda(x)=|x|$ if $|x|$ is large enough (see Figure \ref{functionYoung}).

\begin{figure}
\begin{center}
\includegraphics[height=6.5cm]{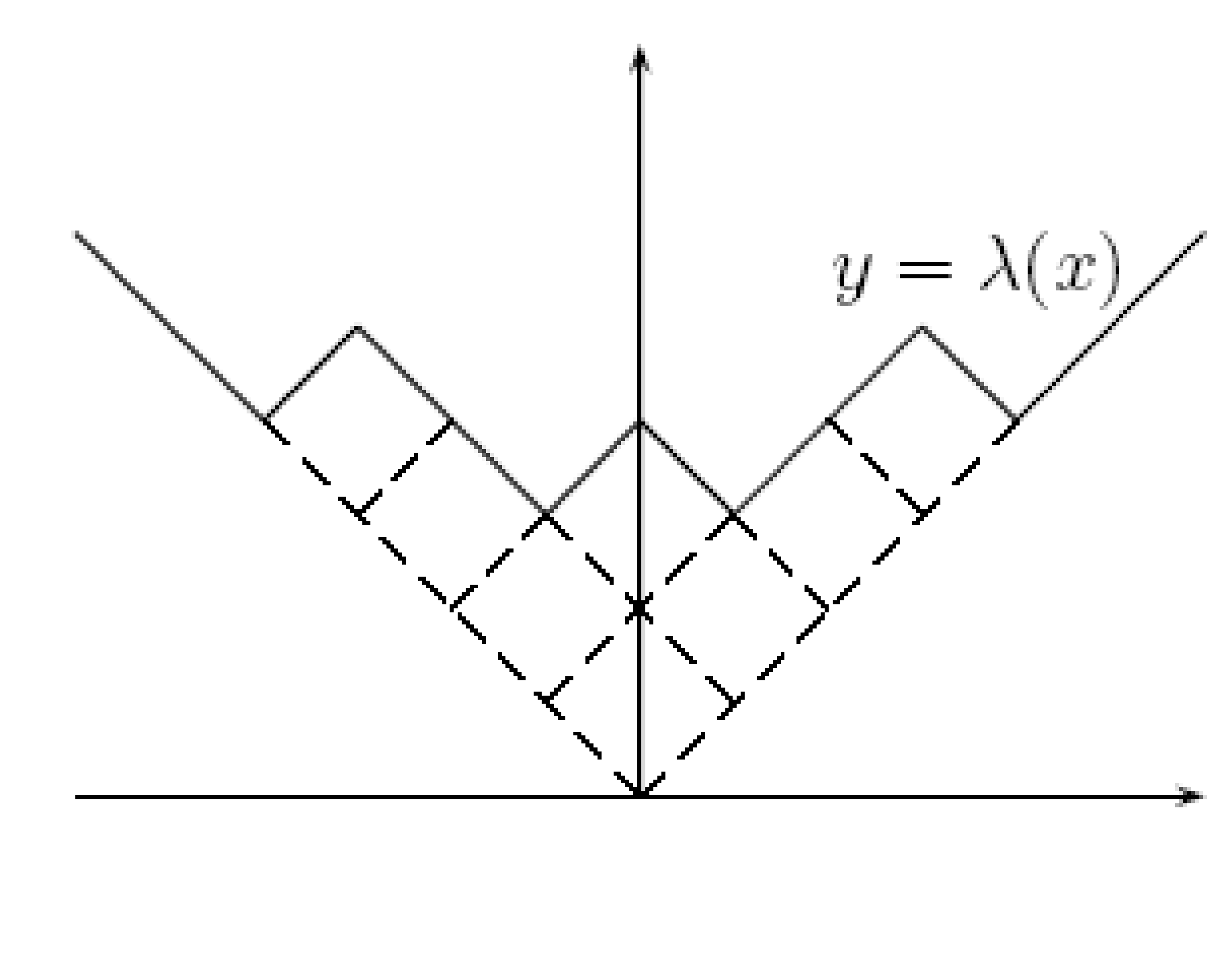}
\caption{The function $\lambda(x)$ corresponding to the Young diagram $\lambda=(4,2,1,1)$.}
\label{functionYoung}
\end{center}
\end{figure}

Suppose
\begin{equation*}
\bar \lambda (x) := \frac{\lambda(x L)}{L}
\end{equation*}
is a normalized boundary of a Young diagram. Let the diagram $\lambda \in \mathbb Y_{[\gamma L^2]}$ be distributed according to the measure $M^{SW}_{[\gamma L^2], L}$. Biane showed that the random function $\bar \lambda (x)$ converges to a deterministic limit function $\lambda_{\gamma} (x)$ (see the exact statement and the formulas for $\lambda_{\gamma} (x)$ in \cite{Biane}). In other words, the random diagram has a limit shape (see Figure \ref{limitshapes}).

\begin{figure}
\center{\includegraphics[height=4cm]{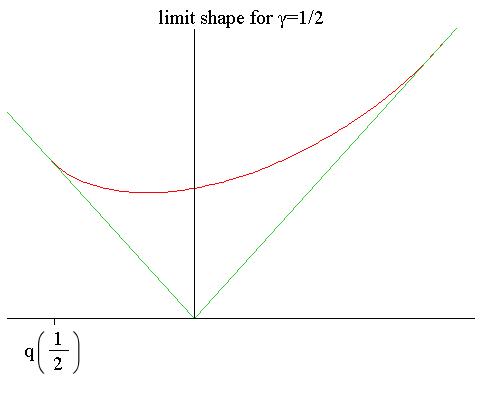}\
\includegraphics[height=4cm]{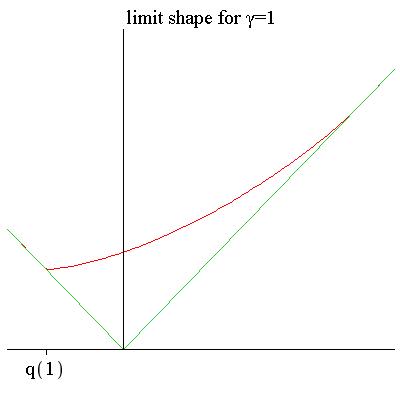}\
\includegraphics[height=4cm]{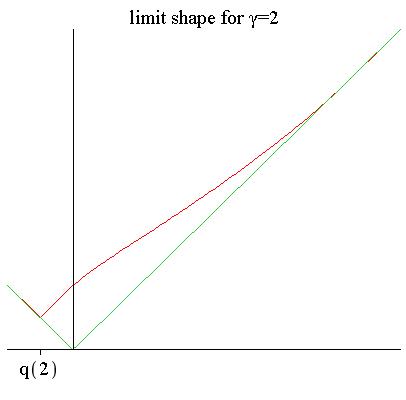}}
\caption{Limit shapes for the measures $P^{\gamma}_L (\lambda)$ and $M_{[\gamma L^2], L}^{SW} (\lambda)$}
\label{limitshapes}
\end{figure}

Let $q(\gamma)$ be the leftmost non-degenerate point of the limit shape $\lambda_{\gamma} (x)$ (see Figure \ref{limitshapes}). It turns out that
\begin{equation*}
\left. \frac{d \lambda'_{\gamma}(x)}{dx} \right|_{x=q(\gamma)+0}= \begin{cases}
+1, &\text{ if $\gamma>1$;}\\
0, &\text{ if $\gamma=1$;}\\
-1, &\text{ if $\gamma<1$.}
\end{cases}
\end{equation*}
The local behaviour of $\bar \lambda (x)$ near the point $q(1)$ in the critical case $\gamma=1$ for the poissonized measure $P_L^1$ was found in \cite{BorOlsh}.

The next possible problem is to find the limit behavior of fluctuations $\bar \lambda(x) - \lambda_{\gamma} (x)$. The first result of this type was obtained by Kerov in \cite{Ker} (see \cite{IvaOls} for a detailed proof) for the Plancherel measures of the symmetric groups.
For $M_{[\gamma L^2], L}^{SW}$, a Kerov type limit theorem was obtained by M\'eliot in \cite{Meliot}. Informally speaking, this result can be stated as follows:
\begin{equation*}
\bar \lambda(x) = \lambda_{\gamma} (x) + \frac{2}{L} \, \Delta_{\gamma} (x), \qquad L \to \infty,
\end{equation*}
where $\Delta_{\gamma} (x)$ is a generalized Gaussian process.

In \cite{BorKuan} local correlations were found in three limit regimes (sine, Airy, Pearcey) for a more general family of measures; these measures correspond to the \emph{two-sided Plancherel} character which arises in the case of nonzero $\gamma^+$ and $\gamma^-$. Moreover, the formulas from \cite{BorKuan} allow one to predict the limit shape of Young diagrams $\lambda^+$ and $\lambda^-$ in the case of two-sided Plancherel characters.

In a different direction, it was shown in \cite{Mkr} that there exists a limit
\begin{equation*}
\lim_{L \to \infty} \frac{- \ln \left( M_{[\gamma L^2,L]}^{SW} (\lambda) \right) }{L}, \qquad \mbox{in probability}.
\end{equation*}
This statement can be viewed as an analog of the Shennon-Macmillan-Breiman theorem.

\subsection{Random height function and GFF}
\label{resultsBorFer}
In this section we give necessary definitions and review results of \cite{BorFer}.

A {\it Gaussian family} is a collection of Gaussian random variables $\{ \xi_a \}_{a \in \Upsilon}$
indexed by an arbitrary set $\Upsilon$. We assume that all the random variables are centered, i.e.
\begin{equation*}
\mathbf E \xi_a = 0, \ \ \ \mbox{ for all } a \in \Upsilon.
\end{equation*}
Any Gaussian family gives rise to a \emph{covariance kernel}
$Cov : \Upsilon \times \Upsilon \to \mathbb R$ defined (in the centered case) by
\begin{equation*}
Cov (a_1, a_2) = \mathbf E ( \xi_{a_1} \xi_{a_2} ).
\end{equation*}

Assume that a function $\tilde C : \Upsilon \times \Upsilon \to \mathbb R$
is such that for any $n\ge 1$ and $a_1, \dots, a_n \in \Upsilon$,
$[\tilde C (a_i,a_j)]_{i,j=1}^{n}$ is a symmetric and positive-definite matrix.
Then (see e.g. \cite{Car}) there exists a centered Gaussian family with the covariance
kernel $\tilde C$.

Let $\mathbb H := \{ z \in \mathbb C : \mathfrak I (z) >0 \}$ be the upper half-plane, and
let $C_0^\infty$ be the space of smooth real--valued compactly supported test functions on
$\mathbb H$.
Let
\begin{equation*}
G(z,w):= -\frac{1}{2 \pi} \ln \left| \frac{z-w}{z - \bar w} \right|, \qquad z,w \in \mathbb H,
\end{equation*}
and define a function $C : C_0^\infty \times C_0^\infty \to \mathbb R$ via
\begin{equation*}
C (f_1, f_2) := \int_{\mathbb H} \int_{\mathbb H} f_1 (z) f_2 (w) G(z,w)
dz d \bar z dw d \bar w.
\end{equation*}

The \emph{Gaussian Free Field} (GFF) $\mathfrak G$ on $\mathbb{H}$ with zero boundary conditions can be
defined as a Gaussian family $\{ \xi_f \}_{ f \in C_0^\infty}$ with covariance kernel $C$.
The field $\mathfrak G$ cannot be defined as a random function on $\mathbb H$,
but one can make sense of the integrals $\int f(z) \mathfrak G(z) dz$ over smooth finite contours
in $\mathbb{H}$ with continuous functions $f(z)$, cf. \cite{She}.

Define the \emph{height function}
\begin{equation*}
H : \mathbb R_{\ge 0} \times \mathbb R_{\ge 1} \times \mathcal P \to \mathbb N
\end{equation*}
as
\begin{equation*}
H (x,y, \{ \lambda^{(n)} \} ) = \sqrt{\pi}\,  \left|\left\{ i\in \{1,2, \dots, [y] \} : \lambda_i^{(y)}  - i + \tfrac 12 \ge x \right\}\right|,
\end{equation*}
where $\lambda_i^{(y)}$ are the coordinates of the signature of length $[y]$ from the
infinite path.
If we equip $\mathcal P$ with a probability measure $\mu_{\gamma}$ then
$H(x,y)$ becomes a random function describing a certain random
stepped surface, or a random lozenge tiling of the half-plane, see \cite{BorFer}.

Define $x(z), y(z) : \mathbb H \to \mathbb R$ via
\begin{equation*}
x(z) = \gamma (1 - 2 \mathfrak R (z)), \ \ y(z) = \gamma |z|^2.
\end{equation*}
Let us carry $H(x,y)$ over to $\mathbb H$ --- define
\begin{equation*}
H^{\Omega} (z) = H ( L x(z), L y(z) ) , \qquad z \in \mathbb H.
\end{equation*}
It is known, cf. \cite{Biane, BorFer}, that there exists a limiting (nonrandom) height function
\begin{equation*}
\tilde h (z) := \lim_{L \to \infty} \frac {\mathbf E H^{\Omega} (z) } {L}, \ \ \ z \in \mathbb H,
\end{equation*}
that describes the limit shape. The fluctuations around the limit shape were studied in
\cite{BorFer}, where it was shown that the fluctuation field
\begin{equation}
\mathcal H (z):= H^{\Omega} (z) - \mathbf E H^{\Omega} (z) , \ \ z \in \mathbb H,
\end{equation}
converges to the GFF introduced above.

In \cite[Theorem 1.3]{BorFer} the following theorem was proved.
\begin{theorem}
\label{thBorFer}

Let $z_1, \dots, z_N \in \mathbb H$ be pairwise distinct complex numbers. Then
\begin{equation*}
\mathbf E ( \mathcal H (z_1) \dots  \mathcal H (z_N) )
\xrightarrow[L \to \infty]{}
\sum_{\sigma \in \mathcal {PM}(N)} \prod_{j=1}^{N/2} G(z_{\sigma(2j-1)}, z_{\sigma(2j)}),
\end{equation*}
where $\mathcal {PM}(N)$ is the set of involutions on $\{ 1,2, \dots, N \}$ with no fixed points also known as perfect matchings (in particular, $\mathcal {PM}(N)$ is empty if $N$ is odd), i.e. all possible disjoint set partitions $\{1,2,\dots,N\}=\{\sigma(1),\sigma(2)\}\sqcup\dots\sqcup\{\sigma(N-1),\sigma(N)\}$.
\end{theorem}


An extension of this theorem was also proved in \cite[Theorems 5.6, 5.8]{BorFer}; that result asserts the convergence for a certain space of test functions.

Let us formulate a similar statement that we prove in this work, and that utilizes
a different space of test functions.

Define a moment of the random height function via
\begin{equation*}
M_{y,k} := \int_{-\infty}^{\infty} x^k (H ( Lx, Ly) - \mathbf E H (Lx, Ly) ) dx.
\end{equation*}
Also define the corresponding moment of the GFF as
\begin{equation*}
\mathcal M_{y,k} = \int_{z \in \mathbb H; y = \gamma |z|^2} x(z)^k \mathfrak G(z) \frac{d x(z)}{dz } dz.
\end{equation*}

\begin{proposition}
As $L \to \infty$, the collection of random variables $ \{ M_{y,k} \}_{y >0, k \in \mathbb Z_{\ge 0}}$ converges, in the sense of finite-dimensional distributions, to  $\{ \mathcal M_{y,k} \}_{y >0, k \in \mathbb Z_{\ge 0}}$.
\end{proposition}
This proposition is a special case of Theorem \ref{mainresult} (see below).

\subsection{Convergence in the sense of states}
\label{algebra}
Consider a probability space $\Omega$ and a sequence of $k$-dimensional random variables
$( \eta_n^1, \eta_n^2, \dots, \eta_n^k )_{n\ge 1}$ on it that converge, in the sense
of convergence of moments, to a Gaussian random vector $(\eta^1, \dots, \eta^k )$
with zero mean. If we define a \emph{state} as

\begin{equation*}
\langle \xi \rangle_{\Omega} := \mathbf E \xi, \qquad \xi \in L^1 (\Omega),
\end{equation*}
then this convergence can be reformulated as
\begin{multline}
\label{Wick1}
\bigl\langle \eta_n^{i_1} \eta_n^{i_2} \dots \eta_n^{i_l} \bigr\rangle_{\Omega} \xrightarrow[n \to \infty]{}
\sum_{\sigma \in \mathcal{PM}(l)} \prod_{j=1}^{l/2} \bigl\langle \eta^{i_{\sigma(2j-1)}} \eta^{i_{\sigma(2j)}} \bigr\rangle_{\Omega},
\\ \mbox{for any $l\ge 1$ and any } \  (i_1, \dots, i_l) \in \{1,2, \dots, k \}^l,
\end{multline}
where, as above, $\mathcal{PM}(l)$ is the set of perfect matchings on $\{ 1,2, \dots, l \}$, i.e. $\{1,\dots,l\}=\{\sigma(1),\sigma(2)\}\sqcup\dots\sqcup\{\sigma(l-1),\sigma(l)\}$.
Indeed, Wick's formula implies that the right-hand side of \eqref{Wick} contains the
moments of $\eta$.

Let $\mathcal A$ be a $*$-algebra and $ \langle\, \cdot \,\rangle$ be a state
(=linear functional taking non\-ne\-ga\-ti\-ve values at elements of the form $a a^*$) on it.
Let $a_1, a_2, \dots, a_k\in\mathcal A$.

Assume that elements $a_1, \dots, a_k$ and the state on $\mathcal A$ depend on
a large parameter $L$, and we also have a $*$-algebra $\mathbf A$ generated by elements
$\mathbf a_1, \dots, \mathbf a_k$ and a state $\phi$ on it. We say that
$(a_1, \dots, a_k)$ converge to $(\mathbf a_1, \dots, \mathbf a_k)$
\emph{in the sense of states} if

\begin{equation}
\label{shod}
\langle a_{i_1} a_{i_2} \dots a_{i_l} \rangle \xrightarrow[ L \to \infty]{} \phi( \mathbf a_{i_1} \dots \mathbf a_{i_l} ),
\end{equation}
and this holds for any $l \in \mathbb N$ and any index set
$(i_1, i_2, \dots, i_l) \in \{1,2, \dots, k \}^l$.

We say that a collection $\{ a_i \}_{i \in \mathfrak J}\subset \mathcal A$ indexed by an arbitrary set
$\mathfrak J$ and depending on a large parameter $L$, converges in the sense of states to a collection
$\{ \mathbf a_i \}_{i \in \mathfrak J}\subset \mathbf A$ if \eqref{shod} holds for any
finite subset of $\{ a_i \}_{i \in \mathfrak J}$ and the corresponding subset in
$\{ \mathbf a_i \}_{i \in \mathfrak J}$.

\subsection{The algebra of shifted symmetric functions}

In this subsection we review some facts about the algebra of shifted symmetric functions, see \cite{OkoOlsh}, \cite{KerOls}, \cite{IvaOls}.

Let $\Lambda^* (n)$ be the algebra of polynomials in $n$ variables $x_1, x_2, \dots$ which become symmetric in new variables

\begin{equation*}
y_i:=x_i-i+\frac12, \qquad i=1,2, \dots, n.
\end{equation*}

The filtration of $\Lambda^*(n)$ is taken with respect to the degree of a polynomial. Define a map $\Lambda^*(n) \to \Lambda^*(n-1)$ as specializing $x_n=0$. The \emph{algebra of shifted symmetric functions} $\Lambda^*$ is the projective limit (in the category of filtered algebras) of $\Lambda^*(n)$ with respect to these maps.

The algebra $\Lambda^*$ is generated by the algebraically independent system $\{ \mathbf p_k \}_{k=1}^{\infty}$, where
\begin{equation*}
\mathbf p_k (x_1, x_2, \dots) := \sum_{i=1}^{\infty} \left( \left( x_i-i+\frac12 \right)^k - \left(-i+\frac12 \right)^k \right), \ \ \ k=1, 2, \dots.
\end{equation*}

Let $\rho, \lambda \in \mathbb Y:=\mathbb Y_0 \cup \mathbb Y_1 \cup \mathbb Y_2 \cup \dots$, and let $r=|\rho|$, $n=|\lambda|$. In the case $r=n$ by $\chi_{\rho}^{\lambda}$ we denote the value of the irreducible character of $S(n)$ corresponding to $\lambda$ on the conjugacy class indexed by $\rho$. In the case $r<n$ by $\chi_{\rho}^{\lambda}$ we denote the value of the same character on the conjugacy class indexed by $\rho \cup 1^{n-r} = (\rho, 1,1, \dots,1) \in \mathbb Y_n$. Define $p_{\rho}^{\#} : \mathbb Y \to \mathbb R$ by
\begin{equation*}
p_{\rho}^{\#} (\lambda) = \begin{cases}
n(n-1) \dots (n-r+1) \dfrac{ \chi^{\lambda}_{\rho}}{ \dim \lambda }, \qquad & n \ge r; \\
0, \qquad & n <r.
\end{cases}
\end{equation*}

Note that elements of $\Lambda^*$ are well-defined functions on the set of all infinite sequences with finitely many nonzero terms.
It turns out that there is a unique element $\mathbf p_{\rho}^{\#} \in \Lambda^*$ whose values coincide with $p_{\rho}^{\#} (\lambda)$ for all $\lambda \in \mathbb Y$ and $x_i = \lambda_i$.
It is known that the set $\{ \mathbf p_{\rho}^{\#} \}_{\rho \in \mathbb Y}$ is a linear basis in $\Lambda^*$. When $\rho$ consists of a single row, $\rho = (k)$, we denote the element $\mathbf p_{\rho}^{\#}$ by $\mathbf p_k^{\#}$. It is known that the set $\{ \mathbf p_k^{\#} \}_{k=1}^{\infty}$ is an algebraically independent system of generators of $\Lambda^*$.

The \emph{weight} of $\mathbf p_{\rho}^{\#}$ is defined by
\begin{equation*}
wt( \mathbf p_{\rho}^{\#}) = |\rho|+l(\rho).
\end{equation*}
Any element $f \in \Lambda^*$ can be written as a linear combination of $\mathbf p_{\rho}^{\#}$'s with nonzero coefficients; the \emph{weight} $wt(f)$ is defined as the maximal weight of $\mathbf p_{\rho}^{\#}$ in this combination. It turns out (see \cite{IvaOls}) that $wt(\cdot)$ is a filtration on $\Lambda^*$. This filtration is called the \emph{weight filtration}.

We will need the following formula (see \cite[Proposition 3.7]{IvaOls}):

\begin{equation}
\label{change}
\mathbf p_k = \frac{1}{k+1} [ u^{k+1}] \left\{ (1 + \mathbf p_1^{\#} u^2 + \mathbf p_2^{\#} u^3 + \dots )^{k+1} \right\} + \text{lower weight terms},
\end{equation}
where ``lower weight terms'' denotes terms with weight $\le k$, and $[u^k]\{ A(u) \}$ stands for the coefficient of $u^k$ in a formal power series $A(u)$.

\section{Statement of the main result}\label{sc:statement}

\subsection{Characters and states on the universal enveloping algebra}
\label{enveloping}
In this subsection we consider a more general approach to the asymptotic analysis of finite trace representations of $U(\infty)$.

Let $I$ be a finite set of natural numbers, and let $U(I) = \left\{ [u_{ij}]_{i,j \in I} \right\}$ be the group of unitary matrices whose rows and columns are marked by elements of $I$.

Let $\mathfrak{gl} (I) = \left\{ (g_{ij})_{i,j \in I} \right\}$ be the complexified Lie algebra of $U(I)$. It is the algebra of all matrices with complex entries and rows and columns indexed by $I$. Let $\mathcal U ( \mathfrak{gl} (I) )$ be the universal enveloping algebra of $\mathfrak{gl} (I)$, and let $Z( \mathfrak{gl} (I) )$ be the center of $\mathcal U ( \mathfrak{gl} (I) )$.
Denote by
\begin{equation*}
\mathcal U (\mathfrak{gl} (\infty)):=\bigcup_{N\ge 1} \mathcal U ( \mathfrak{gl} (\{1,2, \dots, N \}) )
\end{equation*}
the universal enveloping algebra of $\mathfrak{gl} (\infty)=\cup_{N\ge 1}\mathfrak{gl}(\{1,\dots,N\})$.

Denote by $\mathcal D(I)$ the algebra of left-invariant differential operators on $U(I)$ with complex coefficients. It is well known (see e.g. \cite{Zhe}) that there exists a canonical isomorphism
\begin{equation*}
D_I : \mathcal U (\mathfrak{gl} (I) ) \to \mathcal D (I).
\end{equation*}

Let $\chi$ be a character of $U(\infty)$ (see Section \ref{infinite group}), and let $\{ x_{ij} \}$ be the matrix coordinates. Define a \emph{state} $\langle \,\cdot\,\rangle_\chi$ on $\mathcal U( \mathfrak{gl} (\infty) )$ as follows: For any $X \in \mathcal U( \mathfrak{gl} (\infty) )$

\begin{equation}
\label{sost}
\langle X \rangle_{\chi} = D_I (X) \chi (x_{ij}) |_{x_{ij}=\delta_{ij}},
\qquad X \in \mathcal U (\mathfrak{gl} (I) ).
\end{equation}
Note that this definition is consistent for different choices of $I$. In the finite-dimensional case, formula \eqref{sost} gives a (normalized) trace of the image of $X$ in the representation corresponding to $\chi$.


It turns out that computing the state of $X \in Z (\mathfrak{gl} (I) )$ has a probabilistic meaning.

Let $Sign(I)$ be a copy of $\mathbb {GT}_{|I|}$ corresponding to $I$. We shall denote the coordinates of signatures that parameterize irreducible representations of $U(I)$ as $\lambda_1^{I}$, $\lambda_2^{I}$, $\dots$, $\lambda_{|I|}^{I}$.

Similarly to Section \ref{gelfand-tsetlin}, the restriction of $\chi$ to $U(I)$ and its decomposition on the normalized irreducible characters gives rise to a probability measure on $Sign(I)$.

Define the \emph{shifted power sums} $p_{k,I} : Sign(I) \to \mathbb R $ as
\begin{equation*}
p_{k,I} = \sum_{i=1}^{|I|} \left(\lambda_i^I - i + \tfrac12 \right)^k - \left(-i + \tfrac12 \right)^k, \qquad k \in \mathbb N.
\end{equation*}
Let $\mathbb A (I)$ be the algebra of functions generated by the set $\{ p_{k,I} \}_{k=1}^{\infty}$. It is well known that the functions $p_{k,I}$ for $k=1,2 \dots, |I|$ and fixed $I$ are algebraically independent. Therefore, these functions form a system of generators of $\mathbb A (I)$.

It is known (see e.g. \cite{OkoOlsh}) that there exists a canonical isomorphism
\begin{equation*}
Z (\mathfrak{gl} (I) ) \to \mathbb A (I), \ \ \ I \subset \mathbb N, \ \ \ |I| < \infty.
\end{equation*}

For any central element, the value of the corresponding function at a signature corresponds to the scalar operator that this element turns into in the corresponding representation. One shows that the state $\langle X \rangle_{\chi}$ of an element $X \in Z (\mathfrak{gl} (I))$ equals the expectation of the corresponding to $X$ function in $\mathbb A(I)$ with respect to the probability measure on $Sign(I)$.

We identify the functions from $\mathbb A(I)$ and the elements of $Z (\mathfrak{gl} (I) )$ and use the same notation for them.

The correspondence $p_{k,I} \mapsto \mathbf p_{k}$ and the canonical projection $\Lambda^* \to \Lambda^*(|I|)$ give rise to the natural isomorphism between algebras $\mathbb A(I)$ and $\Lambda^*(|I|)$.
Let $p_{\rho,I}^{\#}$ be the functions (and also the elements of $Z (\mathfrak{gl} (I) )$) corresponding to $\mathbf p_{\rho}^{\#}$ with respect to this isomorphism. Note that this isomorphism also induces a \emph{weight filtration} on $\mathbb A(I)$.

By $\{ E_{ij} \}$ we denote the basis of $\mathfrak{gl} (\infty)$ formed by the matrix units.
Let $E_I$ be the matrix consisting of $E_{ij}$, $i,j \in I$. Consider $p_{\rho,I}^{\#}$ as an element of $Z( \mathfrak{gl} (I) )$. It turns out (see \cite{KerOls}) that it can be written in the form
\begin{equation*}
p_{\rho,I}^{\#} = \tr (E_I^{k_1}) \tr (E_I^{k_2}) \dots \tr (E_I^{k_{l(\rho)}}), \ \ \ \rho = (k_1, k_2, \dots, k_{l(\rho)} ).
\end{equation*}

This implies (see \cite[Eq. (4)]{KerOls}) that for $I \subseteq J$ we have
\begin{equation}
\label{dif-vid}
D_J (p_{\rho,I}^{\#}) = \sum_{i_1, \dots, i_k \in I ; \alpha_1, \dots, \alpha_k \in J} x_{\alpha_1 i_1}
\dots x_{\alpha_k i_k} \partial_{\alpha_1 s(1)} \dots \partial_{\alpha_k i_{s(k)}},
\end{equation}
where $k=|\rho|$ and $s \in S(k)$ is an arbitrary permutation with the cycle structure $\rho$.
Formula \eqref{dif-vid} will be crucial for our further computations.

\subsection{The state corresponding to the one-sided Plancherel character}
\label{onePlancherel}
In the present paper we restrict ourselves to the one-sided Plancherel character with a linearly growing parameter. Recall that this character is defined by the formula
\begin{equation}
\label{char-formul}
\chi (U) = \exp \left(\gamma L \sum_{i=1}^{\infty} (x_{ii}-1) \right),
\end{equation}
where $U=[x_{ij}]_{i,j\ge 1}\in U(\infty)$, $\gamma>0$ is a fixed positive number, and $L$ is a growing parameter.

Let $\mu_{\gamma}$ be the probability measure on the path space $\mathcal P$ that corresponds to this character, and let $\langle \cdot \rangle$ be the state on the universal enveloping algebra $\mathcal U (\mathfrak{gl} (\infty))$ that corresponds to this character.

Using \eqref{sost} and \eqref{dif-vid} it is easy to compute the state of elements $p_{\rho,I}^{\#}$:

\begin{equation}
\label{matojidanie}
\langle p_{\rho,I}^{\#} \rangle = |I|^{l(\rho)} \left( \gamma L \right)^{|\rho|} = \gamma^{|\rho|} \left( \frac{|I|}{L} \right)^{|\rho|} L^{l(\rho) + |\rho|}.
\end{equation}

Recall that the family $\{p_{\rho,I}^{\#} \}_{\rho \in \mathbb Y}$ is a linear basis of $Z( \mathfrak{gl} (I) )$. Hence, for $L \to \infty$ and ${|I|}/{L} \to const>0$ we have
\begin{equation}
\label{w-bound}
\langle f \rangle = O(L^{wt(f)}), \qquad f \in Z(\mathfrak{gl} (I)).
\end{equation}
This fact motivates the use of the weight filtration.

\subsection{Main result}\label{sc:mainresult}
In this section we formulate the main result of the paper.

Let $A=\{ a_n \}_{n \geq 1}$ be a sequence of pairwise distinct natural numbers.
Let $\mathcal P_A$ be a copy of the path space $\mathcal P$ corresponding to $A$.
Given $A$, we define the height function
\begin{equation*}
H_A : \mathbb R_{\ge 0} \times \mathbb R_{\ge 1} \times \mathcal P_A \to \mathbb N
\end{equation*}
by setting
\begin{equation*}
H_A \bigl(x,y, {\{ \lambda^{ \{a_1,a_2, \dots,a_n\} } \}}_{n\ge 0}\bigr) = \sqrt{\pi}\left| \left\{i\in \{1,2, \dots, [y] \} : \lambda_i^{ \{ a_1, \dots, a_{[y]} \} }  - i + \tfrac 12 \ge x
\right\}\right|,
\end{equation*}
where $\lambda_i^{ \{ a_1, \dots, a_{[y]} \} }$ denotes the coordinates of the length $[y]$ signature in the infinite path (such notation will be convenient below).
Under the probability measure $\mu_{\gamma}$ on $\mathcal P_A$,
$H_A(x,y, \cdot)=:H_A(x,y)$ becomes a random function on the probability space $(\mathcal P_A, \mu_{\gamma})$.

So far the sequence $A$ is used as a label of the probability space only; these sequences come into play when we consider the joint distributions of several $H_{A_i}$ below. In terms of $U(\infty)$, the choice of $A$ corresponds to the choice of a tower
\begin{equation*}
U(1) \subset U(2) \subset \dots \subset U(\infty),
\end{equation*}
where $U(k) = U(\{a_1, a_2, \dots, a_k \})$ consists of those elements of $U(\infty)$ whose non-trivial matrix elements are placed in the rows and columns marked by elements of $\{a_1, a_2, \dots, a_k\}$.
It is clear that all such towers are conjugate. Therefore, a character of $U(\infty)$ determines the same height function for all choices of $A$.

Let $\{ A_i \}_{i \in \mathfrak J}$ be a family of sequences of pairwise distinct natural numbers
indexed by a set $\mathfrak J$. Introduce the notation
\begin{equation*}
A_i=\{ a_{i,n} \}_{n \ge 1}, \ \ \
A_{i,m}=\{ a_{i,1}, \dots, a_{i,m} \}.
\end{equation*}
Coordinates $a_{i,j} = a_{i,j}(L)$ may depend on the large parameter $L$.

We say that $\{ A_i \}_{i \in \mathfrak J}$ is \emph{regular} if for any
$i,j \in \mathfrak J$ and any $x,y>0$ there exists a limit
\begin{equation}
\alpha(i,x; j,y) = \lim_{L \to \infty} \frac{|A_{i,[xL]} \cap A_{j,[yL]}|}{L}.
\end{equation}

For example, the following family is regular: $\mathfrak{J} = \{ 1,2,3,4 \}$ with $a_{1,n}=n$,
$a_{2,n}= 2n$, $a_{3,n}=2n+1$, and
\begin{equation*}
a_{4,n} = \begin{cases}
n+L, \ \ \ &n =1,2, \dots, L, \\
n-L, \ \ \ &n=L+1, L+2, \dots, 2L, \\
n, \ \ \ &n \ge 2L+1.
\end{cases}
\end{equation*}

Consider the union of copies of $\mathbb H$ indexed by $\mathfrak J$:
\begin{equation*}
\mathbb H(\mathfrak I) := \bigcup_{i \in \mathfrak I} \mathbb H_i.
\end{equation*}
Define a function $C: \mathbb H(\mathfrak J) \times \mathbb H(\mathfrak J) \to \mathbb R \cup \{+\infty \} $
via

\begin{equation*}
C_{ij} (z,w) = \frac{1}{2 \pi} \ln \left|\frac{\alpha(i, |z|^2; j, |w|^2) - z w}{\alpha(i, |z|^2; j, |w|^2) - z \bar w} \right|,
\ \ \ \ i,j \in \mathfrak J, \ z \in \mathbb H_i,\  w \in \mathbb H_j.
\end{equation*}

\begin{proposition}
For any regular family $\mathfrak{J}$ as above, there exists a generalized Gaussian process on
$\mathbb H( \mathfrak J)$ with covariance kernel $C_{ij}(z,w)$. More exactly, for any finite
set of test functions $f_m(z) \in C_0^\infty (\mathbb H_{i_m})$ and $i_1, \dots, i_M \in \mathfrak J$,
the covariance matrix
\begin{equation}
cov( f_k,f_l) = \int_{\mathbb H} \int_{\mathbb H} f_k(z) f_l(w) C_{i_k i_l} (z,w) dz d \bar z dw d \bar w
\end{equation}
is positive-definite.
\end{proposition}
\begin{proof} See \cite[Proposition 1]{Bor}.
\end{proof}

Let us denote this Gaussian process as $\mathfrak G_{ \{ A_i \}_{i \in \mathfrak J} }$.
Its restriction to a single half-plane $\mathbb H_i$ is the GFF with zero boundary conditions introduced above, because
\begin{equation*}
C_{ii} (z,w) = - \frac{1}{ 2 \pi} \ln \left| \frac{z-w}{z - \bar w} \right|, \qquad z,w \in \mathbb H_i, \quad i\in \mathfrak J.
\end{equation*}

As in Section \ref{resultsBorFer}, let us carry $H_A(x,y)$ over to $\mathbb H$ --- define
\begin{equation*}
H_A^{\Omega} (z) = H_A ( L x(z), L y(z) ) , \qquad z \in \mathbb H.
\end{equation*}

As was mentioned above (see Theorem \ref{thBorFer}), the fluctuations
\begin{equation}
\label{imp}
\mathcal H_i (z) := H_{A_i}^{\Omega} (z) - \mathbf E H_{A_i}^{\Omega} (z) , \qquad i \in \mathfrak J, \ z \in \mathbb H_i,
\end{equation}
for any fixed $i$ converge to the GFF. The choice of $A_i$ is not important here because the problem reduces to the case $A = \mathbb N$ by an appropriate conjugation by a permutation matrix.

The main goal of this paper is to study the \emph{joint} fluctuations \eqref{imp} for different
$i$. The joint fluctuations are understood as follows.
Define the moments of the random height function as
\begin{equation}
\label{moment}
M_{i,y,k} := \int_{-\infty}^{\infty} x^k (H_{A_i}( Lx, Ly) - \mathbf E H_{A_i} (Lx, Ly) ) dx.
\end{equation}

It turns out (see \eqref{moment-formula} below) that $M_{i,y,k}$ belongs to $\mathbb A (A_{i, [Ly]})$, and thus it corresponds
to an element of $Z(\mathfrak {gl} (A_{i, [Ly]} ))$ (see Section \ref{enveloping}); denote this element by the same symbol.
Note that all such elements  $M_{i,y,k}$ for all $i,y,k$ belong to the ambient algebra
$\mathcal U (\mathfrak{gl} (\infty))$, and we also have the state $\langle \,\cdot\, \rangle_{\gamma}$ defined
on this ambient algebra (see Section \ref{onePlancherel}). Thus, we can talk about convergence of such elements in the sense of states,
see Section \ref{algebra}.
We are interested in the limit as $L \to \infty$.

We prove that the family $\{ \mathcal H_i \}_{i \in \mathfrak{J} }$ converges to the generalized
Gaussian process $\mathfrak G_{ \{ A_i \} _{ i \in \mathfrak J } }$.
Define the moments of $\mathfrak G_{ \{ A_i \}_{i \in \mathfrak J}}$ by
\begin{equation*}
\mathcal M_{i,y,k} = \int_{z \in \mathbb H;\, |z|^2 = \frac{y}{\gamma} } x(z)^k \mathfrak G_{A_i} (z) \frac{d x(z)}{dz} dz.
\end{equation*}

\begin{theorem}
\label{mainresult}
As $L \to \infty$, for any regular family of sequences $\{ A_i \}_{i \in \mathfrak J}$ the moments
$\{M_{i,y,k}\}_{i \in \mathfrak J, y>0, k \in \mathbb Z_{\ge 0}}$ converge, in the sense of states, to the
moments $\{ \mathcal M_{i,y,k} \}_{i \in \mathfrak J, y>0, k \in \mathbb Z_{\ge 0}}$.
\end{theorem}

Thus, in the $L\to\infty$ limit, the noncommutativity disappears
(limiting algebra $\mathbf A$ is commutative), and yet the random fields $\mathcal H_i$ for different
$i$'s are not independent.

Let $u=L x$. The definition of the height function implies
\begin{equation*}
\frac{d}{d u} H_{A_i} (u, [Ly]) = - \sqrt{\pi} \sum_{s=1}^{[Ly]} \delta \left( u - \left(\lambda_s^{A_{i,[Ly]}} -s + \tfrac 12 \right) \right).
\end{equation*}

Recall that the \emph{shifted power sums} are given by the formula
\begin{equation*}
p_{k,I} = \sum_{i=1}^{|I|} \left( \left(\lambda_i^I - i +\tfrac12 \right)^k - \left(-i +\tfrac12 \right)^k \right) , \ \ \ \ I \subset \mathbb N.
\end{equation*}
One shows that $p_{k,I} \in \mathbb A(I)$, and hence they correspond to
certain elements of $Z(\mathfrak {gl} (I))$ (see Section \ref{enveloping}) that we will denote by the same symbol.

Integrating \eqref{moment} by parts shows that $M_{i,y,k}$ can be rewritten as
\begin{multline}
\label{moment-formula}
\frac{L^{-(k+1)} \sqrt{\pi}}{ k+1} \left( \sum_{s=1}^{[Ly]} \left(\lambda_s^{A_{i,[Ly]}} -s + \tfrac 12 \right)^{k+1}
- \mathbf E \sum_{s=1}^{[Ly]} \left(\lambda_s^{A_{i,[Ly]}} -s + \tfrac 12 \right)^{k+1} \right) \\ = \frac{L^{-(k+1)} \sqrt{\pi}}{ k+1}
( p_{k+1,I} - \mathbf E p_{k+1,I}).
\end{multline}
Thus, Theorem \ref{mainresult} can be reformulated as follows.

\begin{theorem}
\label{th2}
Let $k_1, \dots, k_m \ge 1$ and $I_1, \dots, I_m$ be finite subsets of $\mathbb N$ that
may depend on the large parameter $L$ in such a way that there exist limits
\begin{equation*}
\eta_r=\lim_{L \to \infty} \frac{|I_r|}{L} > 0, \ \ \ \ c_{rs} = \lim_{L \to \infty} \frac{|I_r \cap I_s|}{L},\qquad 1\le r,s\le m.
\end{equation*}

Then, as $L \to \infty$, the collection
\begin{equation*}
\left(L^{-k_r} \left( p_{k_r, I_r}- \mathbf E p_{k_r, I_r} \right) \right)_{r=1}^{m}
\end{equation*}
of elements of\, $\mathcal U (\mathfrak{gl} (\infty))$ converges in the sense of states, cf. \eqref{shod}, to the Gaussian vector
$(\xi_1, \dots, \xi_m)$ with zero mean and covariance
\begin{multline*}
\mathbf {E} \xi_r \xi_s =
\frac{ k_r k_s}{\pi} \oint_{|z|^2 = \frac{\eta_r}{\gamma}; \mathfrak I (z)>0} \oint_{|w|^2 = \frac{\eta_s}{\gamma} ;
\mathfrak I(w)>0}
(x(z))^{k_r-1} (x(w))^{k_s-1} \\ \times
\frac{1}{2 \pi} \ln \left|\frac{c_{rs}/ \gamma - z w}{c_{rs}/ \gamma - z \bar w}\right| \frac{d(x(z))}{dz}
\frac{d(x(w))}{dw} dz dw.
\end{multline*}

\end{theorem}

\section{Computation of covariance}\label{sc:covariance}
In this section we compute the covariance of $p_{k,I_1}$ and $p_{l,I_2}$ (see Theorem \ref{th2}).
At first we find the covariance of $p_{k,I}^{\#}$'s, then the covariance of $p_{k,I}$'s is found with the use of \eqref{change}. The main result of this section is Proposition \ref{prop41}.

Let $I=I(L) \subset \mathbb N$, $I_1=I_1(L) \subset \mathbb N$, $I_2=I_2(L) \subset \mathbb N$
be finite sets such that the following limits exist
\begin{equation*}
\eta = \lim_{L \to \infty} \frac{|I|}{L}, \ \ \eta_1 = \lim_{L \to \infty} \frac{|I_1|}{L}, \ \ \eta_2 = \lim_{L \to \infty} \frac{|I_2|}{L}, \ \
c = \lim_{L \to \infty} \frac{|I_1 \cap I_2|}{L}.
\end{equation*}
Recall that the state $\langle \cdot \rangle$ also depends on $L$, see Section \ref{onePlancherel}.
\begin{lemma}
\label{41}
With the above notations we have
\begin{equation}
\label{lem1}
\left\langle \frac{p_{k,I_1}^{\#} - \langle p_{k,I_1}^{\#} \rangle }{L^k} \cdot \frac{p_{l,I_2}^{\#} -\langle p_{l,I_2}^{\#} \rangle}{L^l} \right\rangle
\xrightarrow[L \to \infty]{} \sum_{n=1}^{\min(k,l)} n \binom{k}{n} \binom{l}{n} c^n \gamma^{k+l-n}.
\end{equation}
\end{lemma}

\begin{proof}
The proof of this Lemma is given in Section \ref{proof41}. It is based on the explicit formula \eqref{dif-vid} for $ D (p_{k,I}^{\#})$.
\end{proof}

Let $u$ and $v$ be formal variables. Let us find the generating function for the right-hand side of \eqref{lem1}.
\begin{lemma}
\label{42}
We have
\begin{equation}
\label{eq42}
\sum_{k=1}^{\infty} \sum_{l=1}^{\infty} \left( \sum_{n=1}^{\min(k,l)} n \binom{k}{n} \binom{l}{n} c^n \gamma^{k+l-n} \right) u^k v^l = \frac{c \gamma u v}{((1 - \gamma u)(1-\gamma v) - c \gamma u v)^2}.
\end{equation}
\end{lemma}

\begin{proof}
Define $\mathbf x_1 = \gamma u$, $\mathbf x_2 = \gamma v$, $\mathbf x_3 = \dfrac{c}{\gamma}$, and let
$\mathbf x$, $\mathbf y$ be formal variables. It is clear that
\begin{equation*}
1 + \sum_{k=1}^{\infty} \sum_{n=1}^k \binom{k}{n} \mathbf y^n \mathbf x^k = \frac{1}{1-(1+\mathbf y) \mathbf x}.
\end{equation*}
Let us differentiate with respect to $\mathbf{y}$ $n$ times. We obtain
\begin{equation*}
\sum_{k=n}^{\infty} \binom{k}{n} \mathbf x^k = \frac{\mathbf x^n}{(1-\mathbf x)^{n+1}}.
\end{equation*}
Therefore, we have
\begin{multline*}
\frac{1}{(1- \mathbf x_1)(1-\mathbf x_2)} + \sum_{n=1}^{\infty}
\sum_{k=n}^{\infty} \sum_{l=n}^{\infty} \binom{k}{n} \binom{l}{n}
\mathbf x_1^k \mathbf x_2^l \mathbf x_3^n =
\frac{1}{(1- \mathbf x_1)(1-\mathbf x_2)} \\ + \sum_{n=1}^{\infty} \frac{ \mathbf x_1^n}{(1-\mathbf x_1)^{n+1}} \frac { \mathbf x_2^n}{(1-\mathbf x_2)^{n+1}} \mathbf x_3^n =
\frac{1}{(1- \mathbf x_1) (1-\mathbf x_2) - \mathbf x_1 \mathbf x_2 \mathbf x_3}.
\end{multline*}
Differentiation with respect to $\mathbf{x_3}$ yields the desired claim.
\end{proof}

In what follows we use certain facts that will only be proved in Section \ref{section5} below. It is convenient to utilize them here though to make the computation of the covariance more transparent.

\begin{lemma}
\label{43}
With the above notations we have
\begin{equation*}
\frac{p_{k,I}^{\#} -  \eta \gamma^k L^{k+1}}{L^k} \xrightarrow[L \to \infty]{} \xi_{k,I}, \qquad
\frac{p_{k,I} - m_{k,\eta} L^{k+1} }{L^k} \xrightarrow[L \to \infty]{} \zeta_{k,I},
\end{equation*}
where $\xi_{k,I}$ and $\zeta_{k,I}$ are Gaussian random variables with zero mean, $m_{k, \eta}$ is a positive constant, and the convergence is understood as convergence in the sense of states (see Section \ref{algebra}).
\footnote{Here the convergence in the sense of states is equivalent to the convergence in moments if we consider $p_{k,I}^{\#}$ and $p_{k,I}$ as random variables on the probability space $Sign(I)$ equipped with the measure $P^{\gamma}_L$, see Section \ref{probMeasure}.}
\end{lemma}
\begin{proof}

The state of $p_{k,I}^{\#}$ is given by \eqref{matojidanie}. The existence of the limit
\begin{equation*}
m_{k, \eta} := \lim_{L \to \infty} \frac{\langle p_{k,I} \rangle}{L^{k+1}}
\end{equation*}
follows from the decomposition of $p_{k,I}$ into a linear combination of $p_{k,I}^{\#}$'s.
Hence the statement of the lemma is a particular case of Proposition \ref{51main} which is proved in Section \ref{section5} below.
\end{proof}

Recall that by $[t^k]\{A(t)\}$ we denote the coefficient of $t^k$ in a formal power series $A(t)$.

Taking into account \eqref{change}, it is easy to see that
\begin{multline*}
m_{k, \eta} = \frac{1}{k+1} [ u^{k+1}] \left\{ \left( 1 + \eta \gamma u^2 + \eta \gamma^2 u^3 + \dots \right)^{k+1} \right\} \\ = \frac{1}{k+1} [ u^{k+1}] \left\{ \left( 1+ \frac{\gamma \eta u^2}{1- \gamma u} \right)^{k+1} \right\}.
\end{multline*}
After computations we obtain
\begin{equation*}
m_{k,\eta} = \frac{1}{k+1} \sum_{r=0}^{k+1} \gamma^{k-r+1} \eta^r \binom{k+1}{r} \binom{k-r}{r-1}.
\end{equation*}
In the case $\eta=1$ this expression coincides with the expression given in
\cite[Prop. 5]{Meliot}.

We do not define $\xi_{k,I_1}$, $\xi_{l,I_2}$ on a common probability space. However, we shall use the notation $\mathbf E( \xi_{k,I_1} \xi_{l,I_2})$ defined by
\begin{equation*}
\mathbf E( \xi_{k,I_1} \xi_{l,I_2}) := \lim_{L \to \infty} \left\langle \frac{ p_{k,I_1}^{\#} - \eta_1 \gamma^k  L^{k+1}}{L^k} \cdot \frac{ p_{l,I_2}^{\#} - \eta_2 \gamma^l L^{l+1}}{L^l} \right\rangle.
\end{equation*}

The existence of the limit in this expression follows from Proposition \ref{51main}. Similarly, denote
\begin{equation*}
\mathbf E( \zeta_{k,I_1} \zeta_{l,I_2}) := \lim_{L \to \infty} \left\langle \frac{ p_{k,I_1} - m_{k, \eta_1} L^{k+1}}{L^k} \cdot \frac{ p_{l,I_2} - m_{l,\eta_2} L^{l+1} }{L^l} \right\rangle.
\end{equation*}

\begin{lemma}
\label{44}
In the notations of the beginning of the section, with formal variables $u$ and $v$, we have
\begin{multline*}
\mathbf E ( \zeta_{k,I_1} \zeta_{l,I_2}) = [u^{k+1}] [v^{l+1}] \left\{ \mathbf E \left( \left( \sum_{i=1}^{\infty} \xi_{i,I_1} u^i \right)
\left( \sum_{j=1}^{\infty} \xi_{j,I_2} v^j \right) \right) \right. \\ \times \left. \left( 1+ \eta_1 u \sum_{i=1}^{\infty} \left( \gamma u \right)^i \right)^k \left( 1+ \eta_2 v
\sum_{j=1}^{\infty} (\gamma v)^j \right)^l \right\} .
\end{multline*}
\end{lemma}

\begin{proof}[Idea of proof]
Recall that the transition formula between generators $\{p_k^{\#} \}$ and $\{p_k \}$ has the following form (see \eqref{change}):
\begin{equation}
\label{change2}
\mathbf p_k = \frac{1}{k+1} [ u^{k+1}] \left\{ (1 + \mathbf p_1^{\#} u^2 + \mathbf p_2^{\#} u^3 + \dots )^{k+1} \right\} + \text{lower weight terms}.
\end{equation}

Informally speaking, Lemma \ref{43} asserts that
\begin{equation*}
p_{k,I}^{\#} = \eta \gamma^{k} L^{k+1} + \xi_{k,I} L^{k} + o(L^k).
\end{equation*}

Consider the expression
$
\langle ( p_{k,I_1} - \langle p_{k,I_1} \rangle ) ( p_{l,I_2} - \langle p_{l,I_2} \rangle ) \rangle.
$
Let us substitute $\eqref{change2}$ into this expression. We obtain that the contributions of order $L^{k+l+2}$ and $L^{k+l+1}$ disappear. The contribution of order $L^{k+l}$ arises iff we choose the component of order $L^r$ from one of the factors corresponding to $p_r^{\#}$, and choose components of the maximal order from the other factors. The statement of the lemma follows from this fact.

A formal proof is given in Section \ref{proof44}.
\end{proof}

From now on we assume that $\eta_1 \le \eta_2$.

Recall that the function $x(z)$ was defined in Section \ref{resultsBorFer}. Let $a$ and $b$ be formal variables.

\begin{lemma}
\label{45}
In the notations of the beginning of the section we have
\begin{equation}
\label{lem4}
\sum_{k,l \ge 1}^{\infty} \frac{ \mathbf E ( \zeta_{k,I_1} \zeta_{l,I_2} ) } {a^{k+1} b^{l+1} } =
\frac{1}{ (2 \pi i)^2 } \oint_{ |z|^2=\frac{\eta_1}{\gamma} } \oint_{ |w|^2 = \frac{\eta_2}{\gamma} } \frac{1}{ a - x(z) } \frac{1}{ b - x(w)}
\frac{c/ \eta_1}{ (c z/ \eta_1 - w)^2 } dz dw,
\end{equation}
where $\dfrac{1}{ a - x(z) }$ and $\dfrac{1}{ b - x(w)}$ are understood as formal power series in $a^{-1}$ and $b^{-1}$, respectively. For $\eta_1 = \eta_2$ we assume that the integration contour in $w$ has the form $|w|^2 = \dfrac{\eta_2}{\gamma} + \delta$, $0< \delta \ll 1$, and the expression in the right-hand side of \eqref{lem4} is understood as the limit as $\delta \to 0$.
\end{lemma}

\begin{proof}
Let us assume now that $u$, $v$ are complex variables, let the contour $\Gamma_u$ be given by
\begin{equation*}
u = \frac{1}{\gamma + r \exp(i \phi)}\, , \qquad \phi \in [0,2 \pi],
\end{equation*}
where $r$ is an arbitrary positive number satisfying $r > 2 \gamma$, and the contour $\Gamma_v$ is a positively oriented circle with center 0 and radius $\epsilon \ll 1$. The contours are chosen in such a way that they encircle the origin (in the counter-clockwise direction), and the formal power series in $u$ and $v$ used below converge inside the contours. The lack of symmetry between the contours $\Gamma_u$ and $\Gamma_v$ is due to the inequality $\eta_1 \le \eta_2$.

Lemma \ref{44} can now be restated as
\begin{multline}
\label{int1}
\sum_{k,l \ge 1} \frac{ \mathbf E (\zeta_{k,I_1} \zeta_{l,I_2} )} {a^{k+1} b^{l+1} } =
\frac{1}{ (2 \pi i)^2} \oint_{u \in \Gamma_u} \oint_{v \in \Gamma_v} \\
\frac{\mathbf E (\sum_{m,n \ge 1} \xi_{m,I_1} \xi_{n,I_2} u^m v^n) }
{(ua - (1+ \eta_1 (\gamma u^2 + \gamma^2 u^3 + \dots))) (vb - (1+ \eta_2 (\gamma v^2 + \gamma^2 v^3 + \dots )} du dv \\ =
\frac{1}{(2 \pi i)^2} \oint_{u \in \Gamma_u} \oint_{v \in \Gamma_v} \frac{1}{(ua - 1 - \frac{ \eta_1 \gamma u^2} {(1-\gamma u)} )}
\frac{1}{(vb - 1 - \frac{ \eta_2 \gamma v^2 }{(1-\gamma v)} )} \\
\times \frac{ c \gamma u v} {((1-\gamma u) (1 - \gamma v) -c \gamma uv)^2} du dv,
\end{multline}
where in the last equality we used Lemmas \ref{41} and \ref{42}). Let us apply a change of variables:

\begin{equation*}
z= - \frac{\eta_1 u}{(1-\gamma u)}, \qquad w = - \frac{1-\gamma v}{\gamma v}.
\end{equation*}

After this change of variables, contours $\Gamma_u$ and $\Gamma_v$ turn into contours $\Gamma_z$
and $\Gamma_w$ which also encircle the origin, and, in addition, $\Gamma_z$ is inside $\Gamma_w$.
Then \eqref{int1} takes the form
\begin{multline*}
\sum_{k,l \ge 1} \frac{\mathbf E (\zeta_{k,I_1} \zeta_{l,I_2})} {a^{k+1} b^{l+1}} = \frac{1}{(2 \pi i)^2}
\oint_{\Gamma_z} \oint_{\Gamma_w} \frac{1}{a+(\eta_1 / z + \gamma(-1+z))}\, \frac{1}{b+(\eta_2 / w +\gamma(-1+w))}
\\ \times \frac{c/\eta_1} {(cz/\eta_1 -w)^2} dz dw.
\end{multline*}

Recall that the expression in the right-hand side is considered as a formal power series in $a^{-1}$ and $b^{-1}$. Thus, the poles of the integrand can be at $z=0$, $w=0$, and at the roots of ${cz}/{\eta_1} -w =0$.
Let us deform $\Gamma_z$ to the circle $|z|^2 = \dfrac{\eta_1}{\gamma}$, and deform $\Gamma_w$ to the circle $|w|^2 = \dfrac{\eta_2}{ \gamma}$ (in the case $\eta_1 = \eta_2$ we deform $\Gamma_w$ to the circle $|w|^2 = \dfrac{\eta_2}{\gamma} + \delta$, where $0< \delta \ll 1$). This deformation does not pass through the poles of the integrand due to conditions $\dfrac{c}{\eta_1} \le 1$ and $\eta_1 \le \eta_2$. Recall that

\begin{equation*}
x(z) = \frac{-\eta_1}{z} - \gamma (z-1), \qquad \mbox{ for $|z|^2  = \dfrac{\eta_1}{\gamma}$},
\end{equation*}

\begin{equation*}
x(w) = \frac{-\eta_2}{w} - \gamma (w-1), \qquad \mbox{ for $|w|^2  = \dfrac{\eta_2}{\gamma}$}.
\end{equation*}

Therefore,
\begin{equation*}
\sum_{k,l \ge 1} \frac{ \mathbf E (\zeta_{k,I_1} \zeta_{l,I_2}) } {a^{k+1} b^{l+1} } =
\frac{1}{(2 \pi i)^2} \oint_{|z|^2=\frac{\eta_1}{\gamma} } \oint_{|w|^2 = \frac{\eta_2}{\gamma}}
\frac{1}{(a-x(z))(b-x(w))} \frac{c/\eta_1 } {(c z / \eta_1 -w)^2} dz dw.
\end{equation*}

In the case $\eta_1 = \eta_2$ we obtain this equation for the $w$-contour of the form $|w|^2 = \dfrac{\eta_2}{\gamma} + \delta$ and take the limit as $\delta \to 0$. The existence of this limit will be clear from the proof of the next proposition.

\end{proof}

\begin{proposition}
\label{prop41}
In the notations of the beginning of the section we have
\begin{multline}
\label{lem5}
\lim_{L \to \infty} \left\langle \frac{ p_{k,I_1} - \langle p_{k,I_1} \rangle}{L^k} \cdot \frac{ p_{l,I_2} - \langle p_{l,I_2} \rangle }{L^l} \right\rangle = \frac{ k l}{\pi} \oint_{|z|^2 = \frac{\eta_1}{\gamma};
\mathfrak I (z)>0}
\oint_{|w|^2 = \frac{\eta_2}{\gamma} ; \mathfrak I (w)>0} \\ \times
(x(z))^{k-1} (x(w))^{l-1}
\frac{1}{2 \pi} \ln \left|\frac{c/ \gamma - z w}{c/ \gamma - z \bar w} \right| \frac{d(x(z))}{dz} \frac{d(x(w))}{dw} dz dw.
\end{multline}
\end{proposition}

\begin{proof}
It is clear that
\begin{equation*}
\frac{1}{a - x(z)} = \frac{1}{a} \left( 1+ \frac{x(z)}{a} + \frac{x^2(z)}{a^2} + \dots \right).
\end{equation*}

Equation \eqref{lem4} implies the equality of the corresponding coefficients in the formal power series:
\begin{equation}
\label{stat}
\mathbf E (\zeta_{k,I_1} \zeta_{l,I_2}) = \frac{1}{(2 \pi i)^2} \oint_{|z|^2 = \eta_1/ \gamma} \oint_{|w|^2 = \eta_2 / \gamma}
x(z)^k x(w)^l \frac{c/\eta_1}{(c z / \eta_1 - w)^2}.
\end{equation}

Note that in the case $\eta_1=\eta_2$ the integrand from right-hand side of \eqref{lem5} has an integrable singularity. Therefore, it suffices to prove that the right-hand side of \eqref{lem5} is equal to the right-hand side of \eqref{stat} for $\eta_1 < \eta_2$.

For $z$ satisfying $|z|^2 = \dfrac{\eta_1}{\gamma}$ we have
\begin{equation*}
2 \ln \left| \frac{ c/\gamma - z w}{ c/ \gamma - z \bar w } \right| =
- \ln \left( \frac{c}{\eta_1} z -w \right) - \ln \left( \frac{c}{\eta_1} \bar z - \bar w \right)
+ \ln \left( \frac{c}{\eta_1} \bar z -w \right) + \ln \left( \frac{c}{\eta_1} z - \bar w \right).
\end{equation*}
Hence, the integral from \eqref{lem5} can be rewritten as
\begin{multline*}
\frac{k l}{(2 \pi i)^2} \oint_{|z|^2 = \frac{\eta_1}{\gamma}} \oint_{|w|^2 = \frac{\eta_2}{\gamma}}
(x(z))^{k-1} (x(w))^{k-1}
\ln \left( \frac{c}{\eta_1} z -w \right) \frac{d(x(z))}{dz} \frac{d(x(w))}{dw} dz dw.
\end{multline*}

Integrating by parts in $z$ and $w$ yields
\begin{equation*}
\frac{1}{(2 \pi i)^2} \oint_{|z|^2 = \frac{\eta_1}{\gamma}} \oint_{|w|^2 = \frac{\eta_2}{\gamma}}
(x(z))^{k} (x(w))^{l}
\frac{c}{\eta_1} \frac{1}{(c z/ \eta_1 -w)^2} dz dw.
\end{equation*}
This expression coincides with \eqref{stat}.

The statement of the proposition is symmetric in $I_1$ and $I_2$; therefore, the condition $\eta_1 \le \eta_2$ can be removed.

\end{proof}

\section{Proof of asymptotic normality}\label{sc:proof}
\label{section5}
Let $I_1= I_1(L)$, $I_2 = I_2 (L)$, $\dots$, $I_r = I_r(L)$ be finite sets of integers that depend on a large parameter $L$ so that the following limits exist:
\begin{gather*}
\eta_i = \lim_{L \to \infty} \frac{|I_i|}{L} ,\qquad
c_{ij} = \lim_{L \to \infty} \frac{|I_i \cap I_j|}{L}, \qquad i,j = 1,2, \dots, r.
\end{gather*}

The main goal of this section is to prove the following statement.
\begin{proposition}
\label{51main}

For any $f_1 \in Z (\mathfrak{gl} (I_1))$, $f_2 \in Z (\mathfrak{gl} (I_2))$, $\dots$,
$f_r \in Z (\mathfrak{gl} (I_r))$ we have
\begin{equation*}
\left( \frac{f_1 - \langle f_1 \rangle }{L^{wt(f_1)-1}}, \frac{f_2 - \langle f_2 \rangle}{L^{wt(f_2)-1}}, \dots, \frac{f_r - \langle f_r \rangle}{L^{wt(f_r)-1}} \right)
\xrightarrow[L \to \infty]{} (\xi_1, \xi_2, \dots, \xi_r),
\end{equation*}
where $(\xi_1, \xi_2, \dots, \xi_r)$ is a Gaussian random vector with zero mean, and the convergence is understood in the sense of states (see Section \ref{algebra}).
\end{proposition}

In Subsection \ref{51section} we prove a particular case of this proposition, in Subsection \ref{proof41} a proof of Lemma \ref{41} is given, in Subsection \ref{53section} we prove Proposition \ref{51main}, and in Subsection \ref{proof44} we prove Lemma \ref{44}.

\subsection{Proof of a particular case of Proposition \ref{51main} }
\label{51section}
In this subsection we prove Proposition \ref{51main} in the case
\begin{equation*}
f_i = p_{k_i,I_i}^{\#}, \qquad i=1,2 \dots, r.
\end{equation*}
Let
\begin{equation*}
\mathcal I = \bigcup_{i=1}^r I_i.
\end{equation*}

Recall that the elements $\{ p_{k_i,I_i}^{\#} \}$ can be written as operators in the algebra of differential operators $\mathbb C[x_{ij}, \partial_{ij}]$, $i,j \in \mathcal I$; they have the following form (see \eqref{dif-vid})
\begin{equation*}
D_{\mathcal I} (p_{k_j, I_j}^{\#}) = \sum_{i_1, \dots, i_{k_j} \in I_j; \alpha_1, \dots \alpha_{k_j} \in \mathcal I}
x_{\alpha_1 i_1} \dots x_{\alpha_{k_j} i_{k_j}} \partial_{\alpha_1 i_2} \dots \partial_{\alpha_{k_j} i_1}.
\end{equation*}
Then a product of $p_{k_j, I_j}^{\#}$ can be viewed as a product of the corresponding differential operators, and the state of the element corresponding to a differential operator $D$ can be computed by the formula
\begin{equation*}
D \exp \left. \left( \gamma L \sum_{i \in \mathcal I} (x_{ii}-1) \right) \right|_{x_{ij}=\delta_{ij}},
\end{equation*}
see Section \ref{enveloping}.

Let $(\xi_1, \xi_2, \dots, \xi_r)$ be a Gaussian random vector with zero mean. Recall that the joint moments of $\xi_i$'s are given by the Wick formula:
\begin{equation}
\label{Wick}
\mathbf E (\xi_{i_1} \xi_{i_2} \dots \xi_{i_l} ) = \begin{cases}
0, \qquad & \mbox{$l$ is odd}, \\
\sum_{\sigma \in \mathcal {PM} (l)} \prod_{i=1}^{l/2} \mathbf E (\xi_{i_{\sigma(2i-1)}} \xi_{i_{\sigma(2i)}}), \qquad & \mbox{$l$ is even},
\end{cases}
\end{equation}
where $\mathcal {PM} (l)$ is the set of perfect matchings on $\{ 1,2, \dots, l \}$.

Define
\begin{equation*}
\nu_j := p_{k_j,I_j}^{\#} - \langle p_{k_j,I_j}^{\#} \rangle, \qquad  j=1, 2, \dots, r,
\end{equation*}
and let
\begin{equation*}
C_{\nu} (i,j) = \lim_{L \to \infty} \frac{\langle \nu_i \nu_j \rangle}{L^{k_i+k_j}}, \qquad i,j =1,2, \dots, r,
\end{equation*}
be the asymptotic covariance of these elements (the existence of the limit is proved below).

\begin{proposition}
\label{52}
In the notations above we have
\begin{equation*}
\frac{\langle \nu_1 \nu_2 \cdots \nu_r \rangle}{L^{k_1 +k_2 + \dots +k_r}} \xrightarrow[L \to \infty]{} \begin{cases}
0, \qquad & \mbox{$r$ is odd}, \\
\sum_{\sigma \in \mathcal {PM} (r)} \prod_{i=1}^{r/2} C_{\nu} (\sigma(2i-1), \sigma(2i)), \qquad  & \mbox{$r$ is even}.
\end{cases}
\end{equation*}
\end{proposition}
This proposition and Wick's formula \eqref{Wick} imply that elements $\dfrac{\nu_i}{L^{k_i}}$, $i=1,2, \dots,r$, are asymptotically Gaussian.

Let us introduce some notation. For a monomial $M$ in $\{ x_{ij}, \partial_{ij} \}$ (i.e. a word in the alphabet $\{ x_{ij}, \partial_{ij} \}$) we define the \emph{support} $supp(M)$ as the set of indices of all letters in the monomial. We call the number of elements in the support the \emph{coverage} and write $cov(M)$. We say that the \emph{$x$-degree} of a monomial is the number of $x$-factors in it, the \emph{$\partial$-degree} is the number of $\partial$-factors, and if the $x$-degree and the $\partial$-degree coincide, we call this number the \emph{degree} of a monomial $M$ and write $deg (M)$. The number of diagonal $\partial$-factors $\partial_{ii}$ in $M$ is called the \emph{capacity} and is denoted by $cap(M)$.

Given two monomials $M_1$ and $M_2$, we say that they are \emph{isomorphic} if there exists a bijection of $supp(M_1)$ and $supp(M_2)$ converting them into each other. By $Isom(M_1)$ we denote the set of monomials with indices from $\mathcal I$ which are isomorphic to $M_1$.

For example, for $M = x_{12} \partial_{23} x_{23} x_{22} \partial_{41} \partial_{33}$ we have $supp(M) = \{ 1,2,3,4 \}$, $cov(M)=4$, $deg(M)=3$, $cap(M)=1$.

Note that for every monomial $M_0$ we have
\begin{equation*}
\langle M_0 \rangle = O\bigl(L^{cap(M_0)}\bigr),
\qquad
\left\langle \sum_{ M: M \in Isom(M_0)} M \right\rangle = O \bigl( |\mathcal I|^{cov(M_0)} L^{cap(M_0)} \bigr).
\end{equation*}

We say that a monomial $M$ is \emph{$\partial$-regular} if for any $i,j \in supp(M)$, $i \ne j$, and for any factor $\partial_{ij}$ in it, there are strictly more letters $x_{ij}$ then $\partial_{ij}$ to the right of this factor. We say that a monomial $M$ is \emph{$x$-regular} if for any $i,j \in supp(M)$, $i \ne j$, and any factor $x_{ij}$ in it, there are strictly more letters $\partial_{ij}$ then $x_{ij}$ to the left of this factor. A monomial $M$ is called \emph{regular} if it is $\partial$-regular and $x$-regular. It is easy to see that if $M$ is not regular then
\begin{equation*}
M \exp \left. \left(\gamma L \sum (x_{ii}-1) \right) \right|_{x_{ij}=\delta_{ij} } = 0.
\end{equation*}

We also extend some of these notions to differential operators of the form
\begin{equation}
\label{firstType}
D = (x_{i_1 i_1}^{l_1} \partial_{i_1 i_1}^{l_1} - \gamma^{l_1} L^{l_1} ) \cdots (x_{i_s i_s}^{l_s} \partial_{i_s i_s}^{l_s} - \gamma^{l_s} L^{l_s} ).
\end{equation}
The \emph{support} of such an operator is the set of indices $\{ i_1, i_2, \dots, i_s \}$, the \emph{coverage} is the number of distinct elements in the support, the \emph{degree} is equal to $l_1 + l_2 + \dots +l_s$.

Let us formulate two lemmas which we prove below.
\begin{lemma}
\label{lemma51}
For any operator of the form \eqref{firstType} we have
\begin{equation*}
\langle D \rangle = O (L^{deg(D) - cov(D)}).
\end{equation*}
Moreover, we have
\begin{equation}
\label{minus1}
\langle D \rangle = O( L^{ deg(D) -cov(D) -1})
\end{equation}
if $\{ i_1, i_2, \dots, i_s \}$ is not a disjoint union of pairs of equal indices (in particular,
\eqref{minus1} holds for odd $s$).
\end{lemma}

\begin{lemma}
\label{lemma52}
Let $C(1), \dots, C(m)$ be monomials of the form
\begin{equation*}
C(l) = x_{\alpha_1^l i_1^l} \dots x_{\alpha_{k_l}^l i_{k_l}^l } \partial_{\alpha_1^l i_2^l} \dots
\partial_{\alpha_{k_l}^l i_1^l}, \ \ 1 \le l \le m,
\end{equation*}
with $cov(C(l)) \ge 2$. \footnote{In the notation $\alpha_m^l$, $i_n^l$ the upper $l$ is an additional index, not a degree.} Assume that the monomial
\begin{equation*}
M = C(1) C(2) \cdots C(m)
\end{equation*}
is regular. Then we have
\begin{equation*}
cap(M) \le deg(M) - cov(M).
\end{equation*}
Moreover, the equality holds iff $m$ is even and there are $m/2$ disjoint sets $J_1, \dots J_{m/2}$ and a partition of $\{ C(1), \dots, C(m) \}$ into pairs $\{ C(j_1), C(j_2) \}$, $\dots$, $\{ C(j_{m-1}), C(j_m) \}$, such that
\begin{equation*}
J_k = supp(C(j_{2k-1})) = supp ( C(j_{2k})), \qquad  k=1,2, \dots ,\frac{m}{2}.
\end{equation*}
\end{lemma}

Let us prove Proposition \ref{52} with the use of Lemmas \ref{lemma51} and \ref{lemma52}. Proofs of the lemmas are given below.

\begin{proof}
Using \eqref{dif-vid} and \eqref{matojidanie}, we can write $\nu_l$ as
\begin{equation}
\label{diag}
\nu_l = \sum_{i \in I_l} (x_{ii}^{k_l} \partial_{ii}^{k_l} - \gamma^{k_l} L^{k_l})  +
\sum_{\alpha_1, \dots, \alpha_{k_l} \in \mathcal I; \ i_1, \dots, i_{k_l} \in I_l}
x_{\alpha_1 i_1} \dots x_{\alpha_{k_l} i_{k_l} } \partial_{\alpha_1 i_2} \dots \partial_{\alpha_{k_l} i_1},
\end{equation}
where the second sum is over the monomials whose coverage is at least 2. Denote the two terms from \eqref{diag} by $\nu_l^{diag}$ and $\nu_l^{off-diag}$, respectively.

Let us open the parentheses in the product
\begin{equation*}
\nu_1 \nu_2 \cdots \nu_r = (\nu_1^{diag} + \nu_1^{off-diag}) \cdots (\nu_r^{diag} + \nu_r^{off-diag} ).
\end{equation*}
We obtain a sum of several terms. Every term is a product of several factors, and each factor is a sum over indices; let us also open the parentheses in each term. In such a way $\nu_1 \nu_2 \cdots \nu_l$ is represented as a sum over indices of products of factors of the form $(x_{ii}^k \partial_{ii}^k - \gamma^k L^k)$ (we call them \emph{diagonal}) and factors of the form $x_{\alpha_1 i_1} \cdots x_{\alpha_k i_k} \partial_{\alpha_1 i_2} \cdots \partial_{\alpha_k i_1}$ whose coverage is greater or equal to 2 (we call them \emph{off-diagonal}).

Consider one term from this sum. Denote by $D_{a_1}, \dots, D_{a_s}$ its diagonal factors coming from $\nu_{a_1}^{diag}, \dots, \nu_{a_s}^{diag}$, $a_1 < \dots <a_s$, and denote by $C_{b_1}, \dots, C_{b_t}$ its off-diagonal factors coming from $\nu_{b_1}^{off-diag}, \dots, \nu_{b_t}^{off-diag}$,
$b_1 < \dots < b_t$, $s+t=r$.

Note that if such a term gives a nonzero contribution to $\langle \nu_1 \nu_2 \dots \nu_r \rangle$, then $C_{b_1} \cdots C_{b_t}$ is regular. Therefore, using Lemma \ref{lemma52} we obtain
\begin{equation*}
cap (C_{b_1} \cdots C_{b_t}) \le k_{b_1} + \dots + k_{b_t} - cov( C_{b_1} \cdots C_{b_t}),
\end{equation*}
where the equality holds iff the factors $C_{b_1}, \dots, C_{b_t}$ can be divided into pairs with disjoint supports for different pairs.

Consider the terms satisfying
\begin{equation*}
supp(D_{a_1} \cdots D_{a_s}) \cap supp(C_{b_1} \cdots C_{b_t}) = \varnothing.
\end{equation*}
From Lemmas \ref{lemma51} and \ref{lemma52} it follows that the contribution of one term is
\begin{equation*}
O( L^{k_1 + \dots + k_r - cov( D_{a_1} \cdots D_{a_s}) - cov( C_{b_1} \cdots C_{b_t}) } ),
\end{equation*}
if both $\{ D_{a_1}, \dots, D_{a_s} \}$ and $\{ C_{b_1}, \dots, C_{b_t} \}$ can be divided into pairs with disjoint supports, and is equal to
\begin{equation*}
O( L^{k_1 + \dots + k_r - cov( D_{a_1} \cdots D_{a_s}) - cov( C_{b_1} \cdots C_{b_t}) -1} )
\end{equation*}
otherwise. In the sum for $\langle \nu_1 \nu_2 \cdots \nu_r \rangle$ there are
\begin{equation*}
O(L^{cov(D_{a_1} \cdots D_{a_s}) + cov(C_{b_1} \cdots C_{b_t}) } )
\end{equation*}
terms of such a form; therefore, their overall contribution has order $L^{k_1 + \dots + k_r}$ only if the supports can be divided into disjoint pairs. It is easy to see that
\begin{multline*}
\langle \nu_1 \nu_2 \rangle = \langle \nu_1^{diag} \nu_2^{diag} \rangle + \langle \nu_1^{off-diag} \nu_2^{off-diag} \rangle
= \\ \sum_{supp(D_1) = supp(D_2) \subset I_1 \cap I_2} \langle D_1 D_2 \rangle +
\sum_{ supp(C_1) = supp(C_2) \subset I_1 \cap I_2} \langle C_1 C_2 \rangle.
\end{multline*}
Hence, the overall contribution of these terms can be written as
\begin{equation*}
\sum_{\sigma \in \mathcal {PM} (r)} \langle \nu_{\sigma(1)} \nu_{\sigma(2)} \rangle \cdots \langle \nu_{\sigma(r-1)} \nu_{\sigma(r)} \rangle
+ O(L^{k_1 + k_2 + \dots +k_r -1}).
\end{equation*}

Now we need to prove that the terms with
\begin{equation}
\label{intersect}
supp(D_{a_1} \cdots D_{a_s}) \cap supp(C_{b_1} \cdots C_{b_t}) \ne \varnothing
\end{equation}
give a contribution of order $O(L^{k_1+k_2 + \dots +k_r - 1})$. Consider one of these terms; denote it by $S$. Consider one of its factors $D_j$ with $supp(D_j) \subset supp(C_{b_1} \cdots C_{b_t})$ and remove $D_j$ from $S$; we obtain an operator $S'$. The state of $S'$ can be estimated as follows: By Lemma \ref{lemma52} we get
\begin{equation*}
cap( C_{b_1} \cdots C_{b_t} ) \le deg( C_{b_1} \cdots C_{b_t} ) - cov (C_{b_1} \cdots C_{b_t} ).
\end{equation*}
The factors of the diagonal type with supports in $supp(C_{b_1} \cdots C_{b_t})$ contribute to the state no more then $L^{deg}$. The product of factors of the diagonal type with supports not in $supp(C_{b_1}, \dots, C_{b_t})$ by Lemma \ref{lemma51} contributes no more than $L^{deg - cov}$.
Therefore,
\begin{equation*}
\langle S' \rangle = O (L^{deg(S') - cov(S')}).
\end{equation*}
Let us estimate the state of the whole $S$. Suppose
\begin{equation*}
D_j = x_{ss}^l \partial_{ss}^l - \gamma^l L^l.
\end{equation*}
Recall that in the algebra $\mathbb C [x_{ij}, \partial_{ij}]$ the following commutation relations hold
\begin{equation}
\label{commut}
[\partial_{ij}, x_{ij}] = \mathbf 1_{(i,j) = (k,l)}.
\end{equation}

Applying these relations we ``move'' all $x_{ss}$'s from $D_j$ to the left and ``move'' all $\partial_{ss}$'s to the right. After this operation $S$ is written as a sum of several terms.
Note that the term in which all $x_{ss}$'s are on the left and all $\partial_{ss}$'s are on the right gives exactly the same contribution as $\gamma^l L^l$ multiplied by $\langle S' \rangle$.
Therefore, the contribution of this term and the contribution of $( - \gamma^l L^l)$ from $D_j$ cancel out. On the other hand, if $x_{ss}$ or $\partial_{ss}$ interact with other factors in the process of ``moving'', then the total number of $\partial_{ss}$ decreases, and the same arguments as in the estimation of $\langle S' \rangle$ show that the contribution of such a term has order
\begin{equation*}
O( L^{k_1 +\dots +k_r - 1 - cov(S)}).
\end{equation*}
The summation over indices contributes $O(L^{cov(S)})$; therefore, the total contribution of terms of the form \eqref{intersect} is equal to $O(L^{k_1 +\dots + k_r -1})$. This completes the proof of Proposition \ref{52}.

In particular, for $r=2$ these arguments imply that $\langle \nu_1 \nu_2 \rangle$ has order $L^{k_1 + k_2}$. The contribution to $L^{k_1+k_2}$ is given by classes of isomorphic graphs. Note that there is a finite number of different classes of non-isomorphic graphs; therefore, there is a limit for the asymptotic covariance, i.e. the limit of $L^{-(k_1+k_2)} \langle \nu_1 \nu_2 \rangle$.

\emph{Proof of Lemma \ref{lemma51}}.
Factors of the form \eqref{firstType} with different indices commute. Hence, it suffices to prove that
\begin{multline}
\label{lemma3'}
\left. \left( x^{l_1} \partial^{l_1} - (\gamma L)^{l_1} \right) \dots \left( x^{l_q} \partial^{l_q} - (\gamma L)^{l_q} \right)
\exp (\gamma L (x-1) ) \right|_{x=1} \\ = \begin{cases}
O (L^{l_1+ \dots + l_q -1}), \ \ \mbox{ for $q =2$}, \\
O (L^{l_1 + \dots + l_q -2}),  \ \ \mbox{ for $q \ge 3$}.
\end{cases}
\end{multline}
Let us open the parentheses in the first factor of the expression
\begin{equation*}
\left( x^{l_1} \partial^{l_1} - (\gamma L)^{l_1} \right) \cdots \left( x^{l_q} \partial^{l_q} - (\gamma L)^{l_q} \right).
\end{equation*}

Applying relations \eqref{commut}, we move all $\partial$ to the right. After this operation several terms appear. The term in which all $\partial$ survive and are on the right gives exactly the same contribution as the factor $\gamma^l L^l$; therefore, the contribution of this term and the contribution of $(-\gamma L)^l$ from the first factor cancel out. If a term arises after $A$ mutual annihilations of $x$ and $\partial$ then it is easy to see that the state of such a term is equal to $O(L^{l_1+ \dots +l_q -A})$. Relation \eqref{lemma3'} follows from that in the case $q=2$. If $q \ge 3$ and $A =1$ then the corresponding term can be written as
\begin{equation*}
\cdots (x^{l_k} \partial^{l_k} - (\gamma L)^{l_k} ) \cdots .
\end{equation*}
Opening the parentheses in the middle factor and applying the same arguments,
we obtain that the contribution of the term without mutual annihilations of $x$ and $\partial$ disappears again. All other terms contribute $O(L^{l_1 + \dots +l_q -2})$; therefore,  \eqref{lemma3'} holds.

\emph{Proof of Lemma \ref{lemma52}.}
We shall code words in the alphabet $\{ x_{ij}, \partial_{ij} \}$ by graphs with an additional structure. To each index from the support of a monomial we assign a vertex of the graph. The edges of the graph can be of two types: $x$-edges and $\partial$-edges. To each letter $x_{ij}$ we assign an oriented $x$-edge between vertices corresponding to indices $i$ and $j$. Similarly, for each letter $\partial_{ij}$ we assign an oriented $\partial$-edge between the same vertices. Besides, we introduce a linear order on the edges of the graph; edge $e_1$ precedes edge $e_2$ if the letter corresponding to $e_1$ is to the right of the letter corresponding to $e_2$. This linear order corresponds to the order in which the factors are applied to the function $\exp \left(\gamma L \sum (x_{ii}-1) \right)$.

For example, the monomial $x_{11} x_{22} x_{12} \partial_{31} x_{23} \partial_{32}$ is coded
by the graph shown in Figure \ref{firstExample}. \footnote{The linear order on the edges is not shown in figures.}

\begin{figure}
\begin{center}
\includegraphics[height=4.0cm]{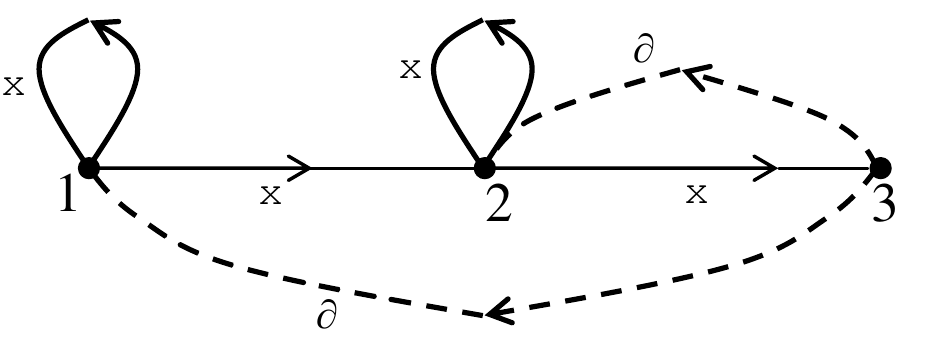}
\caption{The graph corresponding to the monomial $x_{11} x_{22} x_{12} \partial_{31} x_{23} \partial_{32}$ }
\label{firstExample}
\end{center}
\end{figure}

Composing differential operators (concatenating words) corresponds to gluing two graphs at some vertices (that correspond to indices which are common for the two words). The linear order on the edges of the glued graph is induced by the linear orders of the two initial graphs and the condition that the edges of the right monomial (which applied earlier) precede the edges of the left one.

Let us introduce some notations for resulting graphs $G$. Let $V(G)$ be the set of vertices of $G$. Let $E_x(G)$, $E_{\partial} (G)$, $E(G)$ be the sets of \emph{non-degenerate} (connecting different vertices) $x$-, $\partial$- and all edges of $G$, respectively.
Let $L_x(G)$, $L_{\partial}(G)$, $L(G)$ be the sets of \emph{degenerate} edges (loops). Let $E_x^{ab}$, $E_{\partial}^{ab}$, $E^{ab}$ be the sets of (non-degenerate) edges between $a,b \in V(G)$, $a \ne b$. The graph corresponding to a monomial is called $x$-, $\partial$-, or simply \emph{regular} if the initial monomial has the corresponding property.

\emph{Definition.}
We call the graph corresponding to a monomial $x_{\alpha_1 i_1} \dots $ $x_{\alpha_k i_k} \partial_{\alpha_1 i_2} \dots \partial_{\alpha_k i_1}$ with coverage $2 k$ a \emph{complete cycle} of degree $k \ge 1$. We call the graph corresponding to the same monomial but with the weaker condition $2 \le cov \le 2k$ a \emph{cycle} of degree $k \ge 1$. Any cycle can be obtained from a complete cycle by identification of some vertices (but there should be at least two distinct vertices remaining).

\emph{Definition.}
A \emph{complete $x$-cycle} is a graph obtained from a complete cycle by gluing the beginning and the end of each $\partial$-edge; a \emph{complete $\partial$-cycle} is a graph obtained from a complete cycle by gluing the beginning and the end of each $x$-edge.

\emph{Definition.} We call the graph obtained from a complete $x$-cycle by an identification of some (but not all) vertices connected by edges an \emph{$x$-cycle}. Similarly, the graph obtained from a complete $\partial$-cycle by an identification of some (but not all) vertices connected by edges is called a \emph{$\partial$-cycle}.

Examples of these graphs are shown in Figures \ref{CompleteCycle} and \ref{XCycle}.

\begin{figure}
\includegraphics[height=5cm]{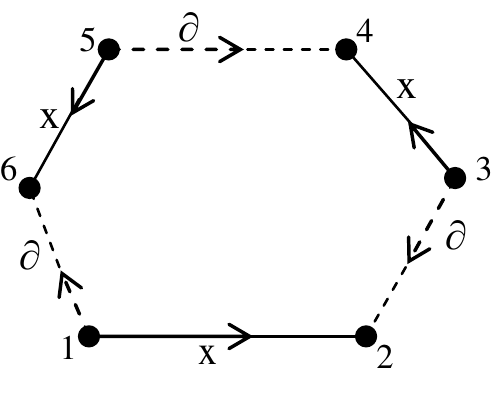}
\includegraphics[height=5cm]{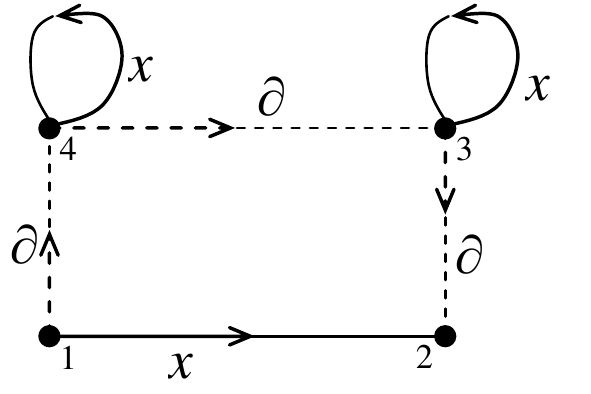}
\caption{The complete cycle corresponding to the monomial $x_{12}  x_{56} x_{34}\partial_{16} \partial_{54}\partial_{32} $ (left) and the cycle corresponding to the monomial $x_{12}  x_{44}x_{33} \partial_{14} \partial_{43}\partial_{32} $ (right).}
\label{CompleteCycle}
\end{figure}

\begin{figure}
\begin{center}
\includegraphics[height=6.5cm]{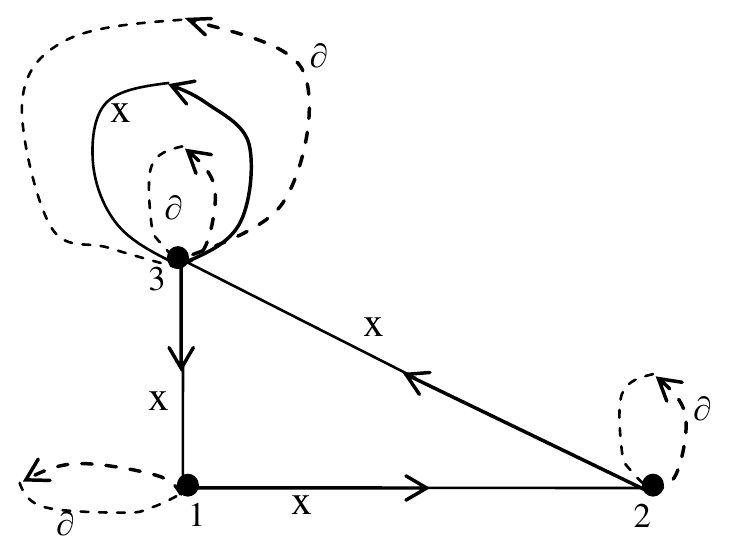}
\caption{The $x$-cycle corresponding to the monomial $x_{12} x_{31} x_{33}x_{23}  \partial_{11}\partial_{33} \partial_{33}\partial_{22}  $}
\label{XCycle}
\end{center}
\end{figure}

Note that if a $\partial$-regular cycle is not an $x$-cycle then
\begin{equation*}
|E_x(G)| > |V(G)|,
\end{equation*}
and for an $x$-cycle we have $|E_x(G)| = |V(G)|$.

It is easy to see that a gluing of a $\partial$-cycle and an $x$-cycle is a regular graph iff closed chains of non-degenerate $\partial$-edges and $x$-edges coincide with each other. Such a gluing of a $\partial$-cycle and an $x$-cycle will be called \emph{regular}.

Denote by $\#$ the operation of gluing graphs, denote by $\#_R$ the regular gluing of a $\partial$-cycle and an $x$-cycle, and denote by $\sqcup$ the "gluing" of graphs with disjoint sets of vertices.

Let us prove the following lemma.
\begin{lemma}
\label{lemma53}
Given a set of cycles $C(1), C(2), \dots, C(n)$, suppose that the graph $G=C(1) \# C(2) \# \dots \# C(n)$ is $\partial$-regular. Then
\begin{equation*}
|V(G)| \le |E_x(G)|.
\end{equation*}
Moreover, the equality holds iff there is a disjoint partition
\begin{equation*}
\{ 1,2, \dots, n \} = \{ i_1, i_2 \} \sqcup \dots \{ i_{2k-1}, i_{2k} \} \sqcup \{ j_1, \dots, j_{n-2k} \},
\end{equation*}
such that $C(i_1), \dots, C(i_{2k-1})$ are $\partial$-cycles,
$C(i_2)$, $\dots$, $C(i_{2k})$, $C(j_1)$, $\dots$, $C(j_{n-2k})$ are $x$-cycles,
and
\begin{equation*}
G = (C(i_1) \#_R C(i_2)) \sqcup \dots (C(i_{2k-1}) \#_R C(i_{2k}) ) \sqcup
C(j_1) \sqcup \dots \sqcup C(j_{n-2k}).
\end{equation*}
\end{lemma}

\emph{Proof of Lemma \ref{lemma53}}.
The $\partial$-regularity of $G$ implies that $C(n)$, $C(n-1) \# C(n)$, $C(n-2) \# C(n-1) \# C(n)$ etc. are $\partial$-regular.
Let us prove the lemma by induction on $n$. The $\partial$-regularity of $C(n)$ implies that $C(n)$ can be obtained by an identification of vertices of a complete $x$-cycle; hence $|V(C(n))| \le |E_x(C(n))|$, and the equality holds iff $C(n)$ is an $x$-cycle.

Assume that the lemma holds for the graph
\begin{equation*}
G_2 := C(2) \# \dots \# C(n)
\end{equation*}
and let us prove it for the graph
\begin{equation*}
G_1 := C(1) \# C(2) \# \dots \# C(n).
\end{equation*}
First, let us prove that $|V(G_1)| \le |E_x(G_1)|$. Consider three cases:

1) $G_1=C(1) \sqcup G_2$, that is, no vertices are glued.

2) $V(G_1) = V(G_2)$, that is, each vertex of $C(1)$ is glued to a vertex of $G_2$.

3) Some vertices of $C(1)$ are glued to vertices of $G_2$, and some vertices of $C(1)$ are not glued to vertices of $G_2$.

In the first case we have
\begin{equation*}
|V(G_1)| - |E_x(G_1)| = |V(G_2)| - |E_x(G_2)| + |V(C_1)| - |E_x(C_1)|.
\end{equation*}
Further, in this case the $\partial$-regularity of $G_1$ implies the $\partial$-regularity of $C(1)$. The induction hypothesis for $G_2$ and the base of induction for $C(1)$ imply the inequality $|V(G_1)| \le |E_x(G_1)|$.

In the second case we have $|V(G_1)| = |V(G_2)|$ and $E_x(G_1) \ge E_x(G_2)$. Therefore, the desired inequality $|V(G_1)| \le |E_x(G_2)|$ follows from the induction hypothesis.

Consider the case (3). It follows from the $\partial$-regularity of $G_1$ that $G_1$ has no non-degenerate $\partial$-edges with the beginning or the end outside of $G_2$. Therefore, new vertices attach to $G_2$ via branches consisting of degenerate and non-degenerate $x$-edges and degenerate $\partial$-edges; the beginning and the end of such a branch lie in $G_2$ and all intermediate vertices lie outside $G_2$. It follows that $G_1$ can be obtained by an identification of some vertices of the graph shown in Figure \ref{FigureProof} (there may be several branches).

\begin{figure}
\center{\includegraphics[height=7cm]{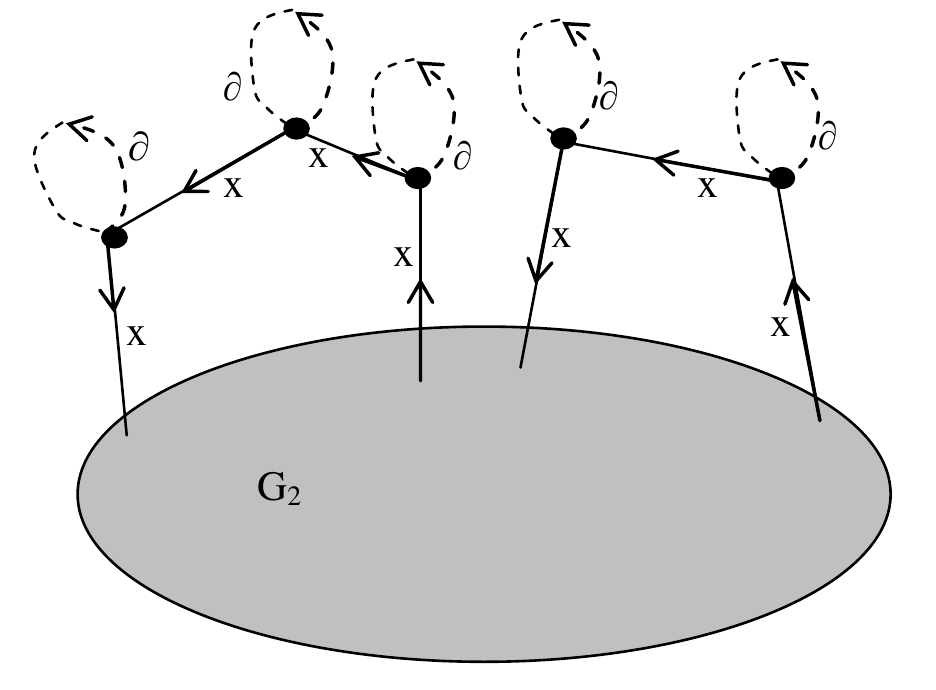}}
\caption{The graph $G_1$ built from the graph $G_2$ with two branches.}
\label{FigureProof}
\end{figure}

The number of vertices of this graph outside $G_2$ is strictly less than the number of non-degenerate $x$-edges with ends in these vertices. It is easy to see that after identifying some of these vertices the number of disappeared vertices is not less than the number of disappeared non-degenerate $x$-edges. Therefore, under the transition from $G_2$ to $G_1$ the number of new vertices is less than the number of new non-degenerate $x$-edges. Hence, by induction hypothesis, we have $|V(G)| < |E_x(G)|$.

We have shown that in all three cases we have $|V(G_1)| \le |E_x(G_1)|$; moreover, in the third case the inequality is strict. Besides, we have shown that
\begin{equation}
\label{ind}
|V(G_1)| - |E_x(G_1)| \le |V(G_2)| - |E_x(G_2)|.
\end{equation}

Assume now that $|V(G_1)| = |E_x(G_1)|$. It follows from \eqref{ind} that $|V(G_2)| = |E_x(G_2)|$ and, by the induction hypothesis, $G_2$ has the desired form.

Consider the same three cases. Due to the strict inequality the case (3) is impossible. In the case (1) the base and the induction hypothesis imply the desired form of $G_1$.

Let us consider the case (2). We have $|V(G_1)|=|V(G_2)|$; therefore, $E_x(C(1)) = \varnothing$. Thus, $C(1)$ is obtained by an identification of vertices in a complete $\partial$-cycle. It follows from the $\partial$-regularity that
\begin{equation*}
|E_{\partial}^{ab} (C(1))| \le |E_x^{ab}(G_2)| - |E_{\partial}^{ab}(G_2)|, \  \ \mbox{for all $a,b \in V(G_2)$, $a \ne b$.}
\end{equation*}

From the structure of $G_2$ (the induction hypothesis) it follows that all elements of $E_{\partial}(C(1))$ must cover non-degenerate $x$-edges of an $x$-cycle from $G_2$. This completes the proof of Lemma \ref{lemma53}.
$\square$

Now we conclude the proof of Lemma \ref{lemma52}. Consider the monomials $C(1)$, $\dots$, $C(n)$ and the corresponding graphs. In terms of a graph, the capacity of a monomial $M$ equals
\begin{equation*}
|L_{\partial}| = deg - |E_{\partial}|.
\end{equation*}

For a regular graph we have $|E_x|=|E_{\partial}|$; therefore, the capacity equals $deg(M)-|E_x|$.
On the other hand, we have $cov(M) = |V|$. Thus, the desired inequality can be written as
\begin{equation*}
cap(M) \le deg(M) - cov(M) \mbox{ is equivalent to } |V| \le |E_x|.
\end{equation*}
This inequality follows from Lemma \ref{lemma53}.

If $cap(M)=deg(M)-cov(M)$ then $|V|=|E_x|$ and we obtain the partition from the statement of Lemma \ref{lemma53}. But for a regular graph the set of indices $\{ j_1, \dots, j_{n-2k} \}$ of single $x$-cycles must be empty (in Lemma \ref{lemma53} we require only $\partial$-regularity). Therefore, we obtain the desired partition of supports into coincident disjoint pairs.

\end{proof}

\subsection{Proof of Lemma \ref{41}}
\label{proof41}
We shall use the notations introduced in Section \ref{51section}.

Let us represent $\nu_i$, $i=1,2$, as a sum of diagonal and off-diagonal terms (see Section \ref{51section}). It is easy to see that
\begin{equation*}
\langle \nu_1 \nu_2 \rangle = \langle \nu_1^{diag} \nu_2^{diag} \rangle + \langle \nu_1^{off-diag} \nu_2^{off-diag} \rangle.
\end{equation*}
The first term can give a nonzero contribution only if the indices in the sums for $\nu_1^{diag}$ and $\nu_2^{diag}$ coincide.

Note that
\begin{equation*}
( x^k \partial^k - (\gamma L)^k) (x^l \partial^l - (\gamma L)^l) \exp( \gamma L (x-1)) |_{x=1}
= k l \gamma^{k+l-1} L^{k+l-1} + o( L^{k+l-1}).
\end{equation*}

The choice of one vertex (the common index in the sums for $\nu_1^{diag}$ and $\nu_2^{diag}$) yields a factor of $|I_1 \cap I_2|$; therefore, the contribution of the first term to $L^{k+l}$ equals
\begin{equation}
\label{di1}
k l \frac{|I_1 \cap I_2|}{L} \gamma^{k+l-1}.
\end{equation}

Consider the second term. The product of monomials from $\nu_1^{off-diag}$ and
$\nu_2^{off-diag}$ gives a nonzero contribution only if the first monomial contains only non-degenerate $\partial$-edges, the second monomial contains only non-degenerate $x$-edges, and the chains of non-degenerate edges coincide. These graphs contain $n \ge 2$ vertices; therefore, the class of isomorphic graphs can contribute to $L^{k+l}$ only if the monomials are a $\partial$-cycle and an $x$-cycle, respectively.

Let us fix an ordered set of $n$ indices. Simple combinatorial computations show that there exist $\binom{k}{n}$ non-isomorphic monomials whose graphs satisfy the following conditions:

1) They are $\partial$-cycles with $n$ non-degenerate edges.

2) They pass through the fixed vertices in a fixed order.

3) They can be obtained after an identification of vertices from a complete $\partial$-cycle of length $k \ge n$.

Similarly, there are $\binom{l}{n}$ non-isomorphic $x$-cycles satisfying similar conditions. Besides, the first vertex of the $x$-cycle can be arbitrary (but the order of visiting other vertices is fixed); this condition contributes a factor of $n$. Finally, the choice of an ordered set of $n$ vertices contributes a factor of $|I_1 \cap I_2|^n + o(L^n)$. Therefore, the contribution of the off-diagonal terms to $L^{k+l}$ equals
\begin{equation}
\label{ge2}
\sum_{n=2}^{\min(k,l)} n \binom{k}{n} \binom{l}{n} \left( \frac{|I_1 \cap I_2|}{L} \right)^n \gamma^{k+l-n},
\end{equation}
where the factor $\gamma^{k+l-n}$ appears because we have $k+l-n$ diagonal $\partial$-operators (represented by $\partial$-loops in a graph).

Combining \eqref{di1} and \eqref{ge2} we obtain the statement of the lemma.

\subsection{ Proof of Proposition \ref{51main}}
\label{53section}

Let $\rho = (\rho_1 \ge \rho_2 \ge \dots \ge \rho_{l(\rho)})$ be a partition. Elements $p_{\rho,I}^{\#}$ correspond to the following operators in the algebra $\mathbb C[x_{ij}, \partial_{ij}]$, $i,j \in \mathcal I$:

\begin{equation*}
D_{\mathcal I} (p_{\rho,I}^{\#}) = \sum_{i_1, \dots, i_{|\rho|} \in I; \alpha_1, \dots \alpha_{|\rho|} \in \mathcal I}
x_{ \alpha_1 i_1} \dots x_{\alpha_{|\rho|} i_{|\rho|} } \partial_{\alpha_1 i_{\sigma(1)}} \dots \partial_{\alpha_{|\rho|}
i_{\sigma(|\rho|)}},
\end{equation*}
where $\sigma \in S(|\rho|)$ is an arbitrary permutation with the cycle structure $\rho$ (see Section \ref{enveloping}).

Note that
\begin{equation*}
\langle D_{\mathcal I} (p_{\rho,I}^{\#}) \rangle = \gamma^{|\rho|} |I|^{l(\rho)} L^{|\rho|} =
\gamma^{|\rho|} \left( \frac{|I|}{L} \right)^{l(\rho)} L^{|\rho| + l(\rho)}.
\end{equation*}

By $wt(\rho) := |\rho|+l(\rho)$ we denote the \emph{weight} of a partition $\rho$. For a permutation $\sigma$ with the cycle structure $\rho$ we also define the \emph{weight} $wt(\sigma) := |\rho|+l(\rho)$. If $|\mathcal I| \sim const \cdot L$ then $\langle D_{\mathcal I} (p_{\rho,I}^{\#}) \rangle$ is proportional to $L^{wt(\rho)}$ as $L \to \infty$.

Recall that the elements $p_{\rho,I}^{\#}$ form a linear basis in $Z (\mathfrak{gl} (I))$. Therefore, for the proof of Proposition \ref{51main} it suffices to prove the asymptotic normality only for these elements.

Let $\rho^{(1)}, \dots, \rho^{(r)}$ be arbitrary partitions.
Let
\begin{equation*}
\mu_j := D_{\mathcal I} \left( p_{\rho^{(j)}, I_j}^{\#} \right) - \left\langle D_{\mathcal I} \left( p_{\rho^{(j)}, I_j}^{\#} \right) \right\rangle, \qquad  1 \le j \le r,
\end{equation*}
and let
\begin{equation*}
C_{\mu} (i,j) = \lim_{L \to \infty} \frac{\langle \mu_i \mu_j \rangle}{L^{wt(\rho^{(i)}) + wt(\rho^{(j)})-2}}
\end{equation*}
be the asymptotic covariance of these elements (the existence of the limit follows from the arguments below).

\begin{proposition}
\label{53}
For any $r \ge 1$ we have
\begin{multline*}
\lim_{L \to \infty} \frac{\langle \mu_1 \mu_2 \cdots \mu_r \rangle}{L^{wt(\rho^{(1)}) +wt(\rho^{(2)}) + \dots +wt(\rho^{(r)}) - r }} \\ =
\begin{cases}
0, \qquad & \mbox{ $r$ is odd}, \\
\sum_{\sigma \in \mathcal{PM} (r)} \prod_{i=1}^{r/2} C_{\mu}(\sigma(2i-1),\sigma(2i)), \qquad & \mbox{ $r$ is even}.
\end{cases}
\end{multline*}
\end{proposition}

Proposition \ref{51main} directly follows from this statement. In the case $l(\rho^{(i)}) = 1$, $1 \le i \le r$, Proposition \ref{53} coincides with already proved Proposition \ref{52}.

The proof of Proposition \ref{53} follows the same lines as the proof of Proposition \ref{52}. Let us formulate two lemmas which generalize Lemmas \ref{lemma51} and \ref{lemma52}.

\begin{lemma}
\label{lemma54}
Consider an operator of the form \footnote{In the expression \eqref{difOpFirstType} and below the notations $k^{a}_b$,$i_{a}^b$ and $\rho_a^b$ are indices while $x^k$ and $\partial^k$ stand for $x$ and $\partial$ raised to the power $k$.}
\begin{multline}
\label{difOpFirstType}
D = \left( x_{i_1^1 i_1^1}^{k_1^1} \dots x_{i_{l_1}^1 i_{l_1}^1}^{k_{l_1}^1} \partial_{i_1^1 i_1^1}^{k_1^1} \dots
\partial_{i_{l_1}^1 i_{l_1}^1}^{k_{l_1}^1} -(\gamma L)^{k_1^1+ \dots + k_{l_1}^1} \right) \cdots \\ \times
\left( x_{i_1^s i_1^s}^{k_1^s} \dots x_{i_{l_s}^1 i_{l_s}^1}^{k_{l_s}^1} \partial_{i_1^s i_1^s}^{k_1^s} \dots
\partial_{i_{l_s}^1 i_{l_s}^s}^{k_{l_s}^s} - (\gamma L)^{k_1^1+ \dots + k_{l_s}^1} \right) .
\end{multline}
With the notation
\begin{equation*}
cov(D) = \left| \{ i_k^j , 1 \le j \le s, 1 \le k \le l_j \} \right|
\end{equation*}
we have
\begin{equation*}
\langle D \rangle = O \left( L^{\sum_{j=1}^s (wt(k^j) -1) - cov(D)} \right).
\end{equation*}
Moreover, we have
\begin{equation*}
\langle D \rangle = O \left( L^{\sum_{j=1}^s (wt(k^j) -1) - cov(D) -1} \right),
\end{equation*}
if the set of supports $\{ \{i_1^j, \dots, i_{l_j}^j \} , 1 \le j \le s \} $ cannot be divided into disjoint pairs such that supports from each pair have non-zero intersection.
\end{lemma}

\begin{lemma}
\label{lemma55}
Let $C(1), \dots, C(m)$ be monomials of the form
\begin{multline*}
C(l) = x_{\alpha_1^l i_1^l} \dots x_{\alpha_{k_l}^l i_{k_l}^l} \partial_{\alpha_1^l i_{\sigma_l(1)}^l }
\dots \partial_{\alpha_{k_l}^l i_{\sigma_l(k_l)}^l }, \quad
\sigma_l \in S(l), \  1 \le l \le m; \ k_1, \dots, k_m \ge 1;
\end{multline*}
with the condition that in $C(l)$ there is at least one letter $x_{ab}$ or $\partial_{ab}$ such that $a \ne b$.
Assume that the monomial $M := C(1) \cdots C(m)$ is regular. Then
\begin{equation*}
cap(M) \le \sum_{j=1}^m (wt(\sigma_j) -1) -cov(M).
\end{equation*}
\end{lemma}
Let us prove Proposition \ref{53} with the use of Lemmas \ref{lemma54} and \ref{lemma55}. Proofs of the lemmas are given below.

Let us write $\mu_j$ in the form
\begin{equation*}
\mu_j = \mu_j^{diag} + \mu_j^{off-diag}, \qquad 1 \le j \le r,
\end{equation*}
where
\begin{equation*}
\mu_j^{diag} = \sum_{i_1, \dots, i_{l_j} } \left( x_{i_1 i_1}^{\rho_1^{(j)}} \cdots x_{i_{l_j} i_{l_j}}^{\rho_{l_j}^{(j)}}
\cdots \partial_{i_1 i_1}^{\rho_1^{(j)}} \cdots \partial_{i_{l_j} i_{l_j}}^{\rho_{l_j}^{(j)}} -
(\gamma L)^{\rho_1^{(j)} + \dots + \rho_{l_j}^{(j)}} \right), \ \ \ l_j:=l(\rho^{(j)}),
\end{equation*}
and $\mu_j^{off-diag}$ consists of all other terms of the expression
\begin{equation*}
\mu_j = D_{\mathcal I} (p_{\rho^{(j)}, I_j}^{\#}) - |I_j|^{l_j} (\gamma L)^{|\rho^{(j)}|}.
\end{equation*}
Let us open the parentheses in the product
\begin{equation*}
\mu_1 \cdots \mu_r = (\mu_1^{diag} + \mu_1^{off-diag}) \cdots (\mu_r^{diag} + \mu_r^{off-diag}),
\end{equation*}
then open the parentheses in each term. We obtain a sum over the united set of indices; terms of this sum are products of factors of two types described in Lemmas \ref{lemma54} and \ref{lemma55}. We call these factors \emph{diagonal} and \emph{off-diagonal}, respectively.

Consider a term such that the supports of its diagonal and off-diagonal factors do not intersect. The estimates of Lemmas \ref{lemma54} and \ref{lemma55} show that the contribution of such a term is $O(L^{ \sum (wt( \rho^{(j)}) - 1) - cov })$ if the supports of factors can be divided into pairs with disjoint supports and is equal to $O(L^{ \sum (wt( \rho^{(j)}) - 1) - cov -1})$ otherwise.
The summation over indices contributes $O(L^{cov})$; therefore, the overall contribution of such terms can be written as
\begin{equation*}
\sum_{\sigma \in \mathcal{PM} (r) } \langle \mu_{i_1} \mu_{i_2} \rangle \cdots \langle \mu_{i_{r-1}} \mu_{i_r} \rangle +
O(L^{\sum(wt( \rho^{(j)}) - 1) -1} ),
\end{equation*}
since it is easy to see that
\begin{equation*}
\langle \mu_1 \mu_2 \rangle = \langle \mu_1^{diag} \mu_2^{diag} \rangle + \langle \mu_1^{off-diag} \mu_2^{off-diag} \rangle.
\end{equation*}

Consider now a term with intersecting supports of its diagonal and off-diagonal factors.
Let us prove that the overall contribution of such a term is $O( L^{\sum (wt( \rho^{(j)}) - 1) - cov -1 })$.

We call two factors of this term \emph{connected} if their supports have non-empty intersection.
Consider a connected component with both diagonal and off-diagonal factors. In this component there is a diagonal factor such that after its removal no new connected components consisting of a single diagonal factor appear (it is easy to see that such a diagonal factor exists). Let us remove this diagonal factor.

In the product of remaining factors we estimate the capacity of the off-diagonal factors with the use of Lemma \ref{lemma55}, we estimate the contribution of the diagonal factors which are not connected with off-diagonal factors with the use of Lemma \ref{lemma54}, and we estimate the contribution of the diagonal factors which are connected with off-diagonal ones as $O(L^{deg})$. Consider the latter factors; note that the number of vertices from supports of such diagonal factors that do not lie in the supports of off-diagonal factors is less than $\sum_{j} (l(\rho^{(j)}) -1)$, where sum is taken over indices corresponding to these diagonal factors.
The contribution of the remaining factors is equal to $O( L^{\sum (wt(\rho^{(j)}) -1) -cov })$, where the sum is taken over the corresponding partitions, and $cov$ is the number of elements in the union of supports of these factors.

Let us return the removed diagonal factor back. Suppose that it corresponds to a partition $\rho$.
Then no more than $l(\rho)-1$ vertices are added to the support. Open the parentheses in this factor; using \eqref{commut}, let us ``move'' all $x_{jj}$'s to the left and ``move'' all $\partial_{jj}$'s to the right. After this operation a sum of several terms appears. Note that the contribution of the term in which all $x_{jj}$'s are on the left and all $\partial_{jj}$'s are on the right and the contribution of $-(\gamma L)^{|\rho|}$ from this factor cancel out. On the other hand, in all other terms the number of diagonal operators $\partial_{jj}$ in off-diagonal factors or connected with them diagonal factors decreases. Repeating the arguments of the previous paragraph we obtain that the state of the whole term can be estimated as
\begin{equation*}
O( L^{\sum_{j=1}^r(wt(\rho^j) -1) -cov -1}),
\end{equation*}
where $cov$ is the coverage of the whole term. This completes the proof of Proposition \ref{53}.

\emph{Proof of Lemma \ref{lemma54}}.
It suffices to prove the lemma under the assumption that the factors of $D$ cannot be divided into two groups with disjoint supports. Therefore,

\begin{equation*}
cov(D) \le \sum_{j=1}^s (l_j-1)+1,
\end{equation*}
and
\begin{equation*}
\sum_{j=1}^s (wt(k^j) -1 ) -cov(D) \ge \sum_{j=1}^s |k^j| -1.
\end{equation*}
Hence, it is enough to prove that
\begin{equation*}
\langle D \rangle = O \left( L^{\sum (|k^j| -1)} \right)
\end{equation*}
for $s=2$, and
\begin{equation*}
\langle D \rangle = O \left( L^{\sum (|k^j| -2)} \right)
\end{equation*}
for $s \ge 3$.

The rest of the argument is completely analogous to the proof of Lemma \ref{lemma51}.

\emph{Proof of Lemma \ref{lemma55}.}
The proof follows similar lines as that of Lemma \ref{lemma52}. Let us introduce some notations generalizing the notations from the proof of Lemma \ref{lemma52}.

A \emph{complete $l$-fold cycle} of degree $k$ is the graph corresponding to a monomial
\begin{equation}
\label{monom}
x_{\alpha_1 i_1} \dots x_{\alpha_k i_k} \partial_{\alpha_1 i_{\sigma(1)} } \dots \partial_{\alpha_k i_{\sigma(k)} },
\end{equation}
whose support consists of $2 k$ distinct vertices, and where $\sigma \in S(k)$ is an arbitrary permutation with $l$ cycles.

An \emph{$l$-fold cycle} of degree $k$ is the graph corresponding to a monomial \eqref{monom} but with the weaker condition $E_x \cup E_{\partial} \ne \varnothing$. Any $l$-fold cycle can be obtained from a complete $l$-fold cycle by an identification of some vertices.

An \emph{$l$-fold $x$-cycle} is an $l$-fold cycle that can be obtained by a gluing of vertices from a complete $l$-fold cycle as follows. In each of $l-1$ connected components we identify all vertices inside a component, and the remaining connected component should be an $x$-cycle after an identification.
Similarly, an \emph{$l$-fold $\partial$-cycle} is an $l$-fold cycle whose $l-1$ connected components each have a single vertex, and the $l$th component is a $\partial$-cycle.

When we say ``multi-fold'' below we mean ``$l$-fold'' with some $l \ge 1$.

Examples of these graphs are shown in Figure \ref{LfoldCycle}.

\begin{figure}
\center{\includegraphics[height=4cm]{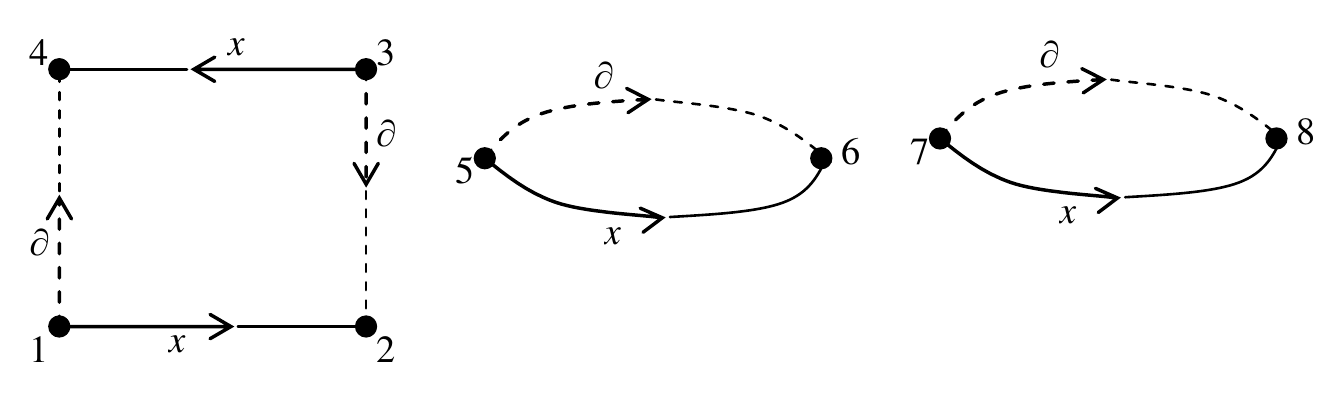}}

\center{\includegraphics[height=4cm]{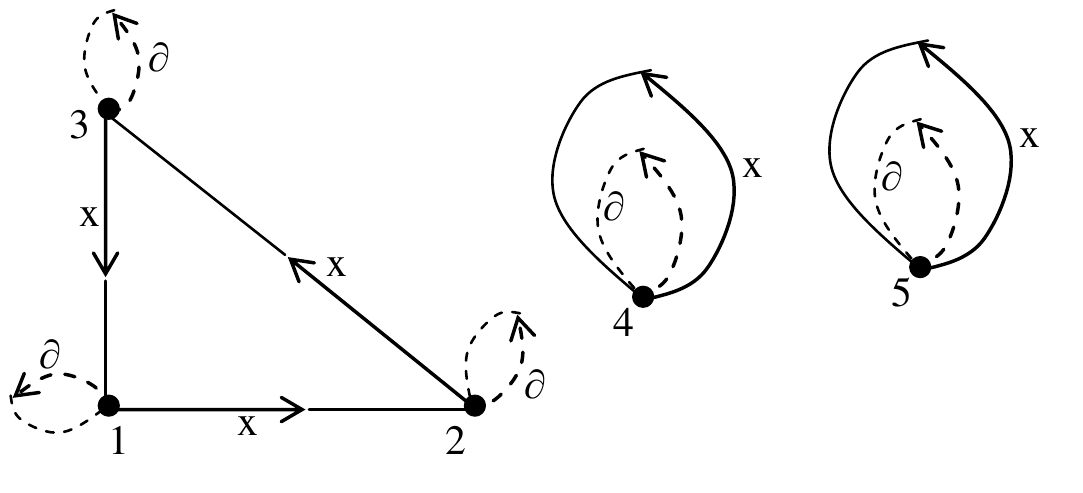}}
\caption{A complete 3-fold cycle of degree 4 (above) and 3-fold $x$-cycle of degree 5 (below).}
\label{LfoldCycle}
\end{figure}

A \emph{regular gluing} of an $l$-fold $\partial$-cycle and an $l$-fold $x$-cycle is an identification of some of their vertices such that the closed chains of non-degenerate $\partial$-edges and $x$-edges coincide with each other and there are no other identifications involved. We denote the regular gluing of graphs by $\#_R$.

Let us formulate a lemma that generalizes Lemma \ref{lemma53}.

\begin{lemma}
\label{lemma56}
Given an $l_1$-fold cycle $C(1)$, an $l_2$-fold cycle $C(2)$, ..., and an $l_n$-fold cycle $C(n)$, suppose that the graph
\begin{equation*}
G = C(1) \# C(2) \dots \# C(n)
\end{equation*}
is $\partial$-regular. Then
\begin{equation*}
|V(G)| \le |E_x(G)| + \sum_{i=1}^n (l_i -1).
\end{equation*}
Moreover, the equality holds iff there is a disjoint partition
\begin{equation*}
\{ 1,2, \dots, n \} = \{ i_1 < i_2 \} \sqcup \{i_3 < i_4 \} \dots \sqcup \{i_{2k-1} < i_{2k} \}
\sqcup \{ j_1, \dots, j_{n-2k} \},
\end{equation*}
such that $C(i_1)$, $\dots$, $C(i_{2k-1})$ are multi-fold $\partial$-cycles, $C(i_2)$, $\dots$, $C(i_{2k})$, $C(j_1)$, $\dots$, $C(j_{n-2k})$ are multi-fold $x$-cycles, and
\begin{equation*}
G = ( C(i_1) \#_R C(i_2) ) \sqcup \dots \sqcup ( C(i_{2k-1}) \#_R C(i_{2k}) ) \sqcup C(j_1) \sqcup \dots \sqcup C(j_{n-2k}).
\end{equation*}
\end{lemma}

\emph{Proof}. Let us prove the lemma by induction on $n$.

\emph{Base of induction.} For $n=1$ the graph $G$ can be obtained from a complete $l$-fold cycle by a gluing of vertices. Let us do this identification in two steps. First, we identify vertices inside connected components, and then we identify vertices of different connected components. On the first step, similarly to the case $n=1$ in Lemma \ref{lemma53}, we obtain that for each component with at least two distinct vertices we have $|V| \le |E_x|$, and for components with only one vertex we have $|V| = |E_x|+1$ ($1=0+1$). After an identification of vertices from different connected components $|V|$ decreases and $|E_x|$ does not increase. Recall that not all components of multi-fold cycles consist of one vertex (by definition); therefore, the lemma holds for $n=1$.

\emph{Induction step.}
Assume that the lemma holds for the graph
\begin{equation*}
G_2 := C(2) \# \dots \# C(n).
\end{equation*}

Consider three cases which are similar to the three cases of Lemma \ref{lemma53}:

1) The supports of $C(1)$ and $G_2$ do not intersect.

2) The supports of $C(1)$ and $G_2$ intersect. For each connected component of $C(1)$ the following is true: All its vertices are glued to some vertices of $G_2$ or all its vertices are not glued to vertices of $G_2$.

3) There is at least one connected component of $C(1)$ such that some vertices of this component are glued to vertices of $G_2$, and some vertices of this component are not glued to vertices of $G_2$.

The case (1) can be treated similarly to the case (1) in the proof of Lemma \ref{lemma53}.
In the case (2) consider the connected component of $C(1)$ with more than one vertex.
If this component lies outside $G_2$ then the $\partial$-regularity of $G$ implies that it is a $\partial$-regular cycle; Lemma \ref{lemma53} gives an estimate $|V| \le |E_x|$ for this component.
If this component lies inside $G_2$ then the number of vertices of $G$ is not greater than the number of vertices of $G_2$, and the number of non-degenerate $x$-edges can increase. On the other hand, one-vertex components increase $|V|$ by at most 1 each (and do not change $|E_x|$).
Summing up over all connected components of $C(1)$, we get
\begin{equation*}
|V(G_1)| - |E_x(G_1)| \le |V(G_2)| - |E_x(G_2)| + (l_1 -1),
\end{equation*}
and the equality is possible only under the conditions:

a) $C(1)$ has $(l_1-1)$ one-vertex components which do not intersect $G_2$.

b) The remaining component of $C(1)$ does not have non-degenerate $x$-edges; its non-degenerate $\partial$-edges must cover $x$-edges from $G_2$.

c) For any $a \ne b \in V(G_2)$ we have
\begin{equation*}
|E_{\partial}^{ab} (C(1))| \le |E_x^{ab} (G_2)| - |E_{\partial}^{ab} (G_2)|.
\end{equation*}
The induction hypothesis and these conditions imply the desired form of $G_2$.

Consider the case (3). The treatment of the case (3) in Lemma \ref{lemma53} shows that for any connected component of $C(1)$ which is partially intersected with $G_2$, we have the strict inequality $|V| < |E_x|$. For all other components the arguments of the case (2) show that the number of new vertices (which they add) is not greater than the number of new non-degenerate $x$-edges plus one. Therefore,
\begin{equation*}
|V(G_1)| - |E_x(G_1)| < |V(G_2)| - |E_x(G_2)| +l_1 -1 \le \sum_{i=1}^n (l_i-1).
\end{equation*}
This completes the proof of Lemma \ref{lemma56}.

Lemma \ref{lemma55} is derived from Lemma \ref{lemma56} in exactly the same way as Lemma \ref{lemma52} is derived from Lemma \ref{lemma53}.

\subsection{Proof of Lemma \ref{44} }
\label{proof44}
 For a set of integers $\vec{i}:=(i_1, \dots, i_{k+1}) \in \mathbb Z^{k+1}_{\ge 0}$, denote by $s(\vec{i})$ the sum of these integers
\begin{equation*}
s(\vec{i}) := \sum_{j=1}^{k+1} i_j,
\end{equation*}
and denote by $n(\vec{i})$ the number of components of $\vec{i}$ that are strictly positive. Any $\vec{i}$ corresponds to a unique partition $\rho(\vec{i})$ which is obtained by removing all zero components and by ordering the remaining ones. Note that
\begin{equation*}
wt(\rho(\vec{i})) = s(\vec{i}) + n(\vec{i}).
\end{equation*}
We agree that $p_0^{\#}=1$.
By \eqref{change} we have
\begin{equation}
\label{change5}
p_{k,I} = \frac{1}{k+1} \sum_{ \vec{i} : s(\vec{i}) +n(\vec{i}) = k+1}
\prod_{j=1}^{k+1} p_{i_j, I}^{\#} + \dots,
\end{equation}
where the dots denote terms with weight $\le k$.

It is known (see \cite[Prop. 4.9]{IvaOls}) that for any partition $\rho = (k_1, k_2, \dots, k_{l(\rho)})$ we have
\begin{equation}
\label{decomp}
\mathbf p_{\rho}^{\#} = \prod_{j=1}^{l(\rho)} \mathbf p_{k_j}^{\#} + \dots,
\end{equation}
where the dots denote terms with weight $\le wt(\rho)-1$.

Recall that the set $\{ \mathbf p_{\rho}^{\#} \}$ is a linear basis in the algebra of shifted symmetric functions. Using \eqref{change5} and \eqref{decomp} we obtain that $p_{k,I}$ decomposes into a linear combination of $p_{\rho,I}^{\#}$'s as follows
\begin{equation}
\label{lineaComb}
p_{k,I} - \langle p_{k,I} \rangle = \sum_{\vec{i}: s(\vec{i}) +n(\vec{i}) = k+1} \left( p_{\rho(\vec{i}), I}^{\#} - \left\langle p_{\rho(\vec{i}), I}^{\#} \right\rangle \right) + \sum_{\rho : wt(\rho) \le k} \left( p_{\rho,I}^{\#} - \left\langle p_{\rho,I}^{\#} \right\rangle \right).
\end{equation}
Consider the product
\begin{equation*}
\left\langle \left( p_{k,I_1} - \langle p_{k,I_1} \rangle \right) \left( p_{l,I_2} - \langle p_{l,I_2} \rangle \right) \right\rangle.
\end{equation*}
Let us substitute \eqref{lineaComb} into this expression and open the parentheses. From Proposition \ref{51main} it follows that
\begin{equation*}
\left\langle \left( p_{\rho_1,I_1}^{\#} - \langle p_{\rho_1,I_1}^{\#} \rangle \right) \left( p_{\rho_2,I_2}^{\#} - \langle p_{\rho_2,I_2}^{\#} \rangle \right) \right\rangle = O \left(L^{wt(\rho_1)+wt(\rho_2)} \right).
\end{equation*}
Hence, a nonzero contribution to the top degree of covariance of $p_{k,I_1}$ and $p_{l,I_2}$ can only be given by terms of the form
\begin{equation}
\label{p-cov2}
\left\langle \left( p_{\rho_1,I_1}^{\#} - \langle p_{\rho_1,I_1}^{\#} \rangle \right) \left( p_{\rho_2,I_2}^{\#} - \langle p_{\rho_2,I_2}^{\#} \rangle \right) \right\rangle,
\end{equation}
where $wt(\rho_1) = k+1$ and $wt(\rho_2) = l+1$.

Let write the elements $p_{\rho,I}^{\#}$ as differential operators. Recall that the graph corresponding to such an operator consists of $l(\rho)$ connected components. From the proof of Lemma \ref{lemma56} it follows that two graphs contribute to the top order of \eqref{p-cov2} only if the following condition holds: They have one common connected component, and all their other components are disjoint and have only one vertex each. Let $\rho_1 = (k_1, k_2, \dots, k_{l(\rho_1)})$ and let $\rho_2=(m_1, m_2, \dots, m_{l(\rho_2)})$. Suppose that the components with common support correspond to $k_{a}$ and $m_{b}$. Then the contribution of these two components equals the covariance of $p_{k_{a}, I_1}^{\#}$ and $p_{m_b, I_2}^{\#}$, and the contributions of all other components are equal to their states.
This completes the proof of Lemma \ref{44}.

\end{document}